\documentclass[a4paper,10pt,reqno]{amsart}

\usepackage[T1]{fontenc}
\usepackage[utf8]{inputenc}
\usepackage[italian, english]{babel}
\usepackage{amsmath, amsfonts, amssymb, amsthm, amscd, mathtools}
\usepackage{geometry}
\usepackage[square,numbers]{natbib}	
\usepackage{bbm, bm}
\usepackage{xspace} 
\usepackage{enumerate, enumitem}
\usepackage{hyperref}
\usepackage{xcolor}
\usepackage{placeins}

\newcommand{\bbF}{{\ensuremath{\mathbb F}} }

\newcommand{\bbP}{{\ensuremath{\mathbb P}} }

\newcommand{\bbW}{{\ensuremath{\mathbb W}} }

\newcommand{\cB}{{\ensuremath{\mathcal B}} }
\newcommand{\cC}{{\ensuremath{\mathcal C}} }

\newcommand{\cE}{{\ensuremath{\mathcal E}} }
\newcommand{\cF}{{\ensuremath{\mathcal F}} }

\newcommand{\cK}{{\ensuremath{\mathcal K}} }
\newcommand{\cL}{{\ensuremath{\mathcal L}} }
\newcommand{\cM}{{\ensuremath{\mathcal M}} }

\newcommand{\cP}{{\ensuremath{\mathcal P}} }

\newcommand{\cR}{{\ensuremath{\mathcal R}} }
\newcommand{\cS}{{\ensuremath{\mathcal S}} }
\newcommand{\cT}{{\ensuremath{\mathcal T}} }

\newcommand{\R}{\mathbb{R}}

\newcommand{\N}{\mathbb{N}}

\renewcommand{\P}{\mathbb{P}}
\newcommand{\E}{\mathbb{E}}

\newcommand{\ind}{\ensuremath{\mathbf{1}}}

\DeclarePairedDelimiterX{\inprod}[2]{\langle}{\rangle}{#1, #2}

\theoremstyle{plain}
\newtheorem{theorem}{Theorem}[section]
\newtheorem*{theorem*}{Theorem}
\newtheorem{lemma}[theorem]{Lemma}
\newtheorem*{lemma*}{Lemma}
\newtheorem{proposition}[theorem]{Proposition}
\newtheorem*{proposition*}{Proposition}
\newtheorem{corollary}[theorem]{Corollary}

\theoremstyle{definition}
\newtheorem{definition}{Definition}[section]
\newtheorem{assumption}{Assumption}[section]

\theoremstyle{remark}
\newtheorem{remark}{Remark}[section]
\newtheorem*{remark*}{Remark}

\definecolor{darkviolet}{rgb}{0.58, 0.0, 0.83}

\definecolor{gre}{rgb}{0.03,0.50,0.03}

\numberwithin{equation}{section}

\hypersetup{pdftitle={MFG\_to\_MFC},pdfauthor={Andrea Amato, Federico Cannerozzi, Giorgio Ferrari},hidelinks,%
pdfstartview={XYZ null null 1.4},colorlinks=true,linkcolor=blue,urlcolor=blue}

\newcommand{\mail}[1]{\href{mailto:#1}{\normalfont\texttt{#1}}}

\makeatletter
\def\@setthanks{\vspace{-\baselineskip}\def\thanks##1{\@par##1\@addpunct.}\thankses}
\makeatother

\begin{document}

\title[Mean-field singular control]{On Mean-field Singular Stochastic Control Problems}
\author[A.~Amato]{Andrea~Amato\textsuperscript{\MakeLowercase{a},1}}
\thanks{\noindent \textsuperscript{a} Dipartimento di Matematica, Universit\`a di Bologna, Bologna, Italy.}
\author[F.~Cannerozzi]{Federico~Cannerozzi\textsuperscript{\MakeLowercase{b},2}}
\thanks{\noindent \textsuperscript{b} Bielefeld University, Center for Mathematical Economics (IMW), Bielefeld (Germany).}
\author[G.~Ferrari]{Giorgio~Ferrari\textsuperscript{\MakeLowercase{b},3}}
\thanks{\noindent
\noindent \textsuperscript{1} E-mail: \mail{andrea.amato9@unibo.it}.
\\
\noindent \textsuperscript{2} E-mail: \mail{federico.cannerozzi@uni-bielefeld.de}.
\\
\noindent \textsuperscript{3} E-mail: \mail{giorgio.ferrari@uni-bielefeld.de}}

\date{\today}

\begin{abstract}
We study a class of mean-field control (MFC) problems with singular controls over a finite horizon, allowing for general dependence of the cost functional on the measure argument. We derive an auxiliary mean-field game (MFG) with singular controls, which we refer to as a potential MFG, and show that, under suitable convexity assumptions, any solution to this potential MFG yields a solution to the original MFC problem.
We apply this general result to a version of the classical Monotone Follower Problem by I.\ Karatzas and S.\ E.\ Shreve (\emph{SIAM J.\ Control Optim.} 22(6), pp.\ 856–877, 1984) with scalar mean-field interaction.
The associated potential MFG with singular controls is solved by exploiting its connection with optimal stopping for the optimization step and by a suitable application of the Kakutani–Fan–Glicksberg fixed-point theorem. In the case of strategic complementarities, the mean-field equilibrium (and hence the optimal policy of the original MFC problem) is characterized by a continuous nonincreasing free boundary that uniquely solves a nonlinear integral equation.
To the best of our knowledge, this is the first paper to provide a complete characterization of the optimal policy in a finite-horizon mean-field singular stochastic control problem.
\end{abstract}

\maketitle

{\small 

\noindent \textbf{Keywords:} singular stochastic control; mean-field control; mean-field games; mean-field monotone follower problem; free boundary.

\smallskip

\noindent \textbf{AMS 2020:} } 49N80, 65D15, 91A16, 93E20.

\smallskip

\bigskip


\section{Introduction}
\label{into}

In recent years, the study of mean-field games (MFGs) and mean-field control (MFC) problems
has become central to the analysis of stochastic control systems with a large number of interacting
agents, where individual behavior is coupled through the empirical distribution of states and/or
controls. For an overview of methodologies, techniques, and applications, we refer to the two-volume monograph \cite{carmona2018probabilistic_vol1, carmona2018probabilistic_vol2}. However, the existing literature has primarily focused on MFGs and MFC problems with regular (classical) control strategies, and comparatively little is known about models involving singular controls or optimal stopping, in particular about the structure of their solutions.

In regular control settings, it is well known that a certain class of MFGs, known as ``potential mean-field games'', can be solved by studying an auxiliary MFC problem (the earliest mention of this relation appears already in the seminal work \cite[Sect.\ 2.6]{LasryLions}; see also \cite{Graber} for a recent detailed review of potential MFGs). This connection has proved useful for establishing existence results, equilibrium selection, and learning procedures (see \cite{Cardaliaguet1,Cardaliaguet2}, among others, for learning algorithms, and \cite[Sect.\ 3]{Graber} and the references therein for a review of the selection problem). More recently, the work \cite{hofer2024optimal} showed, in a general non-Markovian setting, that any solution to a given MFC problem induces a mean-field equilibrium for an associated MFG, in which the running and terminal cost functions appearing in the representative player’s objective are derived from the cost functional of the MFC problem and from its linear derivatives with respect to the measure variable. 

In this paper, we consider a class of mean-field singular stochastic control problems and, under suitable convexity requirements, we obtain a somewhat \emph{vice versa} result to that achieved in \cite{hofer2024optimal} in the setting of regular control problems. More precisely, we show that, given a mean-field singular stochastic control problem whose data satisfy suitable growth, regularity, and convexity assumptions, it is possible to derive an MFG with singular controls such that any equilibrium of this MFG yields a solution to the corresponding MFC problem with singular controls. We believe that this result has two notable consequences. First, as we demonstrate in a mean-field version of the classical Monotone Follower Problem of \cite{karatzas1984monotone} (see below for a discussion on this contribution of the paper), the established relation between MFC problems and MFGs paves the way for the characterization of optimal policies in mean-field singular stochastic control problems.
Second, it implies uniqueness of mean-field equilibria of the auxiliary MFG whenever the original MFC problem admits a unique solution, which is guaranteed in problems in which the performance criterion is strictly convex/concave with respect to the singular control variable.

Characterizing equilibria in MFGs with singular controls is comparatively simpler than characterizing solutions to mean-field singular stochastic control problems, since in MFGs the flow of measures $\mu$, representing the distribution of the states (and possibly the actions) of the other players, is given and fixed. Consequently, the first step in the solution of the auxiliary MFG consists of solving a singular stochastic control problem parametrized by $\mu$, followed by a fixed-point argument. Although this two-step approach introduces an additional fixed-point problem, it is more tractable for characterizing optimal solutions than directly addressing the MFC problem with singular controls. Indeed, when the latter is approached via dynamic programming techniques, one would need to analyze a variational inequality on the space of probability measures, study its regularity, and construct a solution to a Skorokhod reflection problem on an infinite-dimensional space (see, e.g., \cite{Guo-etal}). When approached via the Pontryagin maximum principle, one would again need to study a reflected BSDE with an endogenously determined reflection condition depending on the law of the reflected process.

As already anticipated, the aforementioned connection MFC-MFGs is tested on a mean-field version of the seminal Monotone Follower Problem in \cite{karatzas1984monotone}. Here, a decision maker aims to track a Brownian trajectory via a nondecreasing process (the monotone follower), with the goal of minimizing an expected cost functional over a finite time horizon. This intertemporal cost functional consists of the cumulative expected cost of exerting control, a linear terminal cost, and the time integral of a quadratic running cost function penalizing the misplacement of the current controlled state level with respect to $\alpha$ times its current average, for some parameter $\alpha \in \R$ measuring the strength of the strategic interaction.

According to the general recipe, we derive the associated auxiliary MFG of singular controls and prove the existence of a mean-field equilibrium. This is achieved through a suitable application of the Kakutani–Fan–Glicksberg fixed-point theorem to the best-reply map, which we show to be well defined and to map a subset of $L^2([0,T])$ into itself, endowed with the topology induced by weak convergence. Furthermore, the mean-field equilibrium is unique, since it coincides with the unique optimizer of the MFC problem.

The unique mean-field equilibrium is shown to be characterized by a moving free boundary, which depends on time and on the time-dependent mean-field parameter. Notably, 
when the interaction parameter satisfies
$\alpha \in (0,2)$, 
we can extend our analysis and prove that, for any given and fixed continuous mean-field parameter $(\theta_t)_{t\in [0,T]}$, the free boundary, continuously extended at the terminal time, is the unique solution to a nonlinear integral equation within a suitable class of continuous and nonincreasing functions of time. This result allows us to derive a system of functional equations that uniquely identifies the mean-field equilibrium. Indeed, the integral equation for the free boundary is coupled with the consistency condition requiring that, at each time, the equilibrium mean-field parameter coincides with the expected value of the optimally singularly controlled Brownian trajectory. To the best of our knowledge, a similar characterization of the optimal policy in mean-field singular stochastic control problems appears here for the first time.

A simple iterative scheme then allows us to plot the equilibrium control (which also coincides with the optimal control for the original MFC problem) and its expected value. Notice that this iterative scheme converges to the unique equilibrium since, in the case $\alpha \in (0,2)$ considered here, the MFG exhibits strategic complementarity and the best-reply map is therefore monotone increasing (see also \cite{dianetti2020nonzero,dianettiferrarifischer,dianetti2023unifying} for papers on MFGs with strategic complementarities).

\smallskip

\textbf{Related Literature.} Here we provide a review of the literature on MFGs and MFC problems with singular controls that is relevant to our study.

Abstract existence results for solutions to MFGs with singular controls, in general frameworks allowing for extended formulations, have been obtained in \cite{DenkertHorst,Fu,FuHorst} by means of topological fixed-point theorems, and in \cite{dianetti2023unifying} via lattice-theoretic arguments, in settings that may also feature common noise. Several papers have also addressed the problem of characterizing equilibria in specific examples arising from applications. In this regard, we mention \cite{Aid-etal,Campi-etal,CannerozziFerrari,Cao-etal,CaoGuo,Christensen-etal,DianettiLQ,DianettiFerrariTzouanas,Ferrari-Tzouanas,GuoXu}, where questions related to optimal investment in one-dimensional or Markov-modulated one-dimensional settings have been studied, possibly also in stationary frameworks. Finally, a class of degenerate MFGs with singular controls arising from the relaxation and entropy-regularization of MFGs of optimal stopping is studied in \cite{Dumitrescu-etal}, with the aim of the theoretical development of reinforcement-learning algorithms.

Recently, MFC problems with multidimensional singular controls and nonlinear jump impacts have been studied in \cite{DenkertHorst-MFC}, where a dynamic programming principle is derived for the value function and, under additional regularity assumptions, the value function is shown to solve an appropriate variational inequality in the space of probability measures. The work \cite{Bo-etal} studies mean-field control problems with singular controls under general dynamic state-control-law constraints. Using a relaxed control formulation and compactification arguments, the existence of optimal controls is established. By treating the controlled McKean-Vlasov dynamics as an infinite-dimensional constraint, the problem is reformulated to derive a stochastic maximum principle, a constrained BSDE via Lagrange multipliers, and results on uniqueness and stability of the associated constrained FBSDE.

In \cite{CannerozziFerrari}, a specific one-dimensional ergodic MFC problem of irreversible investment is studied, and comparisons between different notions of equilibrium are performed through closed-form solutions; namely, comparisons between Nash equilibrium (i.e.\ an MFG solution), Pareto efficiency (i.e.\ an MFC solution), and coarse-correlated equilibrium (a refinement of the Nash equilibrium concept).
Particularly closely related to the present work is also \cite{cannerozzi2026stationary}, where, in the case of an Ornstein-Uhlenbeck process with ergodic cost functional of quadratic type, solutions to the ergodic stationary MFC problem are shown to be in bijection with solutions to the associated potential MFG.
The present paper focuses instead on the MFG-to-MFC direction, which is established in a considerably broader framework with respect to the mean-field dependence of the cost functional. We stress that this latter direction is particularly useful for obtaining explicit characterizations of MFC solutions, as shown by the study of the mean-field monotone follower problem.

Finally, necessary and sufficient Pontryagin maximum principles for mean-field singular stochastic control problems have been obtained in \cite{Hafayed1,Hafayed2,ShiWu}, among others. In those works, the adjoint equation takes the form of a mean-field backward stochastic differential equation (BSDE). We also refer to the introduction of \cite{Hafayed1} for further references on mean-field BSDEs.

\smallskip
\textbf{Structure of the paper.} 
The rest of the paper is organized as follows: Section \ref{sec:preliminaries} gathers the notation and some preliminary notions about differentiability of functions of probability measures.
In Section \ref{sec:general} we state and prove the general result connecting mean-field singular control problems and singular mean-field games, while in Section \ref{sec:application} we focus on the mean-field version of the Monotone Follower Problem.
In particular, Section \ref{sec:ctrl_pb} solves the control problem for fixed interaction term, while the fixed-point step is addressed in Section \ref{sec:fixed_point}. Finally, Section \ref{sec:free_boundary} deals with the strategic complementarity case $\alpha \in (0,2)$ and characterizes the free boundary, continuously extended at the terminal time, as the unique solution to an integral equation.


\section{Preliminaries}\label{sec:preliminaries}

\subsection*{Notation}
Let $n \geq 1$, $d \geq 1$ be integers. We denote by $\R^{n\times d}$ the set of $n\times d$ matrices with real entries.
For $A \in \R^{n\times d}$, we denote by $A^\top$ its transpose.
We denote by $\cP_2(\R^n)$ the set of probability measures $\mu$ over $\R^n$ whose second moment is finite, i.e.\ $\int_{\R^n} \vert y \vert ^2 \mu(dy) < \infty $.
We equip $\cP_2(\R^n)$ with the topology generated by the 2-Wasserstein distance on $\cP_2(\R^n)$ (see, e.g., \cite[p. 352]{carmona2018probabilistic_vol1} for the definition).
For $T >0$, we define by $\cM([0,T];\cP_2(\R^n))$ the set of measurable maps $\mu:[0,T] \to \cP_2(\R^n)$. We refer to any $\mu=(\mu_t)_{t \in [0,T]} \in \cM([0,T];\cP_2(\R^n))$ as a measurable flow of measures.

\subsection*{Differentiability of Functions of Probability Measures}
For the reader's convenience, we recall here some known definitions and facts about the differentiability of functions of measures. The following notions can be found in \cite[Chapter 5]{carmona2018probabilistic_vol1}.

\smallskip
We say that a continuous function $\phi:\cP_2(\R^n) \to \R$ is linearly differentiable if there exists a function $\delta_\mu \phi: \cP_2(\R^n) \times \R^n \to \R$ so that
\[
\phi(\mu) - \phi(\nu) = \int_0^1 \int_{\R^n} \delta_\mu \phi( t \mu + (1-t)\nu )(y) (\mu - \nu)(dy)dt,
\]
with $\delta_\mu \phi(\mu)(y)$ being jointly continuous, of at most quadratic growth in $y \in \R^n$ uniformly in $\mu$ for $\mu \in \cK$, where $\cK$ is any bounded subset of $\cP_2(\R^n)$.

\smallskip
We say that a continuous function $\phi:\cP_2(\R^n) \to \R$ is L-differentiable at $\mu_0 \in \cP_2(\R^n)$ if there exists a measurable function $\partial_\mu \phi(\mu_0): \R^n \to \R^n$ such that, for any $\mu \in \cP_2(\R^n)$, for any atomless Polish probability space $\big( \tilde \Omega, \tilde{\cF}, \tilde \P\big)$, for any random variables $ \tilde X_0$ with law $\mu_0$ and $ \tilde{X}$ with law $\mu$, it holds
\begin{equation}\label{eq:def:L_derivative}
    \phi(\mu) - \phi(\mu_0) = \tilde{\E}[\partial_\mu \phi(\mu_0)( \tilde{X}_0)(\tilde{X}-\tilde{X}_0)] + o ( \Vert \tilde{X} - \tilde{X}_0 \Vert_{L^2}).
\end{equation}
We say that $\phi:\cP_2(\R^n) \to \R$ is L-differentiable if it is differentiable at any point $\mu \in \cP_2(\R^n)$.

\smallskip
We say a function $h:\R^n \times \cP_2(\R^n) \to \R$ is continuously jointly differentiable in the linear derivative sense (respectively, in the L-derivative sense) if the partial derivatives $\partial_x h(x,\mu)$ and $\delta_\mu h(x,\mu)(y)$ (respectively, $\partial_\mu h(x,\mu)(y)$) exist and are continuous with respect to the product topologies.

\smallskip
We say that a function $\phi:\R^n \times \cP_2(\R^n) \to \mathbb{R}$ is L-jointly convex if, for any $(x,\mu)$ and $(x',\mu')$ in $\R^n \times \cP_2(\R^n)$, it holds
\begin{equation} \label{eq:L-convexity}
    \phi(x, \mu) - \phi(x', \mu') \leq \partial_x \phi(x, \mu) (x - x') + \tilde{\E} \left[ \partial_\mu \phi(x, \mu)(\tilde X) \big(\tilde X - \tilde X'\big)\right],
\end{equation}
for any atomless Polish probability space $\big(\tilde{\Omega},\tilde{\cF},\tilde{\P}\big)$ and $\tilde{X}$, $\tilde{X}'$ random variables with law $\mu$, $\mu'$ respectively.

\section{The Mean-field Singular Stochastic Control Problem}\label{sec:general}

Let $T > 0$ be a fixed time horizon.
Let $n \geq 1$, $d \geq 1$, and $m \geq 1$ be integers.
Let $(\Omega, \cF, \bbF \coloneqq (\cF_t)_{t \geq 0}, \bbP)$ be a complete filtered probability space, with $\bbF$ satisfying the usual assumptions, on which a $d$-dimensional $\bbF$-Brownian motion $W$ is defined.
We assume that $\bbF$ is the usual $\P$-augmentation of the natural filtration generated by the Brownian motion $W$.

\begin{definition}
An admissible singular control is a process $\xi:[0,T] \times \Omega \to ([0,+\infty))^m \subseteq \R^m$ so that, for any $i=1,\dots,m$, $\xi^i$ is  $\bbF$-adapted, non-decreasing, càdlàg and $\xi^i_{0-} = 0$ $\P$-a.s., and such that $\E[ \vert \xi_T \vert ^2 ] < \infty$.
We denote by $\cB$ the set of admissible controls.
\end{definition}

\smallskip
Let $b:[0,T] \times \R^n \to \R^n$, $\sigma:[0,T] \times \R^n \to \R^{n \times d}$ and $\zeta:[0,T] \to \R^{n \times m}$  be measurable functions.
For any control $\xi \in \cB$, we consider the following dynamics:
\begin{equation}\label{eq:dynamics}
    d X^\xi_t = b\big(t,X^\xi_t\big) dt + \sigma \big(t, X^\xi_t\big) d W_t + \zeta(t)d \xi_t, \quad X^\xi_0 = x_0.
\end{equation}
Let $C: [0,T] \times \R^n \times \cP_2(\R^n) \to \R$ and $G: \R^n \times \cP_2(\R^n) \to \R$ be measurable functions. Let $K:[0,T] \to ([0,\infty))^m $ be continuous.
We associate to the dynamics \eqref{eq:dynamics} the following mean-field control problem: find the control that minimizes the cost functional
\begin{equation}\label{eq:cost_MFC}
   J(\xi) \coloneqq \E\left[ \int_{0}^{T} C\big(t, X^\xi_t, \cL(X^\xi_t)\big) dt + G\big(X^\xi_T, \mathcal{L} \big(X^\xi_T \big)\big) + \int_0^T K(t)d\xi_t \right],
\end{equation}
under the dynamics constraint \eqref{eq:dynamics}, where $\cL(X^{\xi}_t)$ stands for the law of the random variable $X^{\xi}_t$.

\begin{remark}
In the cost functional \eqref{eq:cost_MFC}, the Stieltjes integral with respect to $\xi$ is taken over the closed interval $[0, T]$. Thus, any jump of $\xi$ at the terminal time $T$ contributes to the value of the integral. With a slight abuse of notation, throughout the remainder of the paper, we adopt the convention $\int_0^T K(t)d\xi_t := \int_{[0, T]} K(t)d\xi_t$.
\end{remark}

\medskip
We make the following assumptions:
\begin{assumption}\label{assumptions:general_part}

\begin{enumerate}[label=(\roman*)]
    \item $b,\sigma$ are continuously differentiable and Lipschitz continuous in $x \in \R^n$ uniformly in $t \in [0,T]$. 
    \item $\zeta(t)$ is a continuous function.
    \item The functions $\mu \mapsto C(t,x,\mu)$, $\mu \mapsto G(x,\mu)$ are linearly differentiable with jointly continuous linear derivatives $\delta_\mu C(t,x,\mu)(y)$ and $\delta_\mu G(x,\mu)(y)$.
    \item The maps $C(t,x,\mu)$, $G(x,\mu)$ and the linear derivatives $\delta_\mu C(t,x,\mu)(y)$ and $\delta_\mu G(x,\mu)(y)$ are jointly continuous and with at most quadratic growth, in the sense that
    \begin{equation*}
        \vert C(t,x,\mu) \vert + \vert G (x,\mu) \vert + \vert \delta_\mu C(t,x,\mu)(y) \vert + \vert \delta_\mu G(x,\mu)(y) \vert \leq \kappa \Big( 1 + \vert x \vert^2 + \int_{\R^n} \vert z \vert^2 \mu(dz) + \vert y \vert^2 \Big),
    \end{equation*}
    for some positive constant $\kappa$.
    \item The partial derivatives $\partial_x C(t,x,\mu)$, $\partial_x G (x,\mu)$, $\partial_y \delta_\mu C(t,x,\mu)(y)$ and $\partial_y \delta_\mu \allowbreak G(x,\mu)(y)$ exist, are jointly continuous, and have at most linear growth, in the sense that
    \begin{multline*}
        \vert \partial_x C(t,x,\mu) \vert + \vert \partial_x G (x,\mu) \vert + \vert \partial_y\delta_\mu C(t,x,\mu)(y) \vert + \vert \partial_y\delta_\mu G(x,\mu)(y) \vert 
        \leq \kappa \Big( 1 + \vert x \vert + \left( \int_{\R^n} \vert z \vert^2 \mu(dz)\right)^\frac{1}{2} + \vert y \vert \Big).
    \end{multline*}
\end{enumerate}
\end{assumption}
Notice that Assumption \ref{assumptions:general_part}(i) implies that the partial derivatives $\partial_x b(t,x)$ and $\partial_x\sigma \allowbreak (t, x)$ are uniformly bounded in $(t,x) \in [0,T] \times \R^n$.
Moreover, $(v)$ of Assumption \ref{assumptions:general_part} implies that $C(t,x,\mu)$ and $G(x,\mu)$ are also L-differentiable, with L-derivative (in the sense of equation \eqref{eq:def:L_derivative}) given by $\partial_\mu C(t,x,\mu)(y) = \partial_y \delta_\mu C(t,x,\mu)(y)$ and $\partial_\mu G(x,\mu)(y) \allowbreak = \partial_y \delta_\mu G(x,\mu)(y)$, by \cite[Proposition 5.48]{carmona2018probabilistic_vol1}.

\bigskip
We now introduce the associated potential MFG, in the same spirit as in \cite{hofer2024optimal}.
Let $c:[0,T] \times \R^n \times \cP_2(\R^n) \to \R$ and $g:\R^n \times \cP_2(\R^n) \to \R$ be given by
\begin{equation}
\begin{aligned} \label{eq:MFG:new_functions}
    c(t,x, \mu) &= C(t,x, \mu) + \int_{\R^n} \delta_\mu C(t,\tilde x, \mu) (x) \mu(d\tilde x), \\ 
    g(x, \mu) &= G(x, \mu) + \int_{\R^n} \delta_\mu G(\tilde x, \mu) (x) \mu(d\tilde x).  
\end{aligned}
\end{equation}
For $\mu=(\mu_t)_{t \in [0,T]} \in \cM([0,T];\cP_2(\R^n))$, we consider the following cost functional:
\begin{equation}\label{eq:MFG:cost}
    \begin{aligned}
    J_{g}(\xi,\mu)  \coloneqq \E\left[ \int_{0}^{T} c\big(t,X^\xi_t, \mu_t \big) dt + g\big(X^\xi_T, \mu_T \big)  + \int_0^T K(t)d\xi_t \right],
    \end{aligned}
\end{equation}
under the same dynamics constraint \eqref{eq:dynamics}.

\begin{definition}\label{def:MFG:solution}
We say that a pair $(\xi^*,\mu^*)$, with $\xi^* \in \cB$ and $\mu^* \in \cM([0, T]; \allowbreak \cP_2(\R^n))$, is a solution to the potential MFG if the following two properties hold:
\begin{enumerate}[label=(\roman*)]
    \item $J_g(\xi^*, \mu^*) \leq J_g(\xi, \mu^*)$, for any admissible control $\xi \in \cB$, and
    \item $\mu^*_t = \cL\big(X^*_t\big)$, $t\in [0,T]$, where $X^*$ denotes the solution to equation \eqref{eq:dynamics} associated to the optimal control $\xi^*$.
\end{enumerate}
\end{definition}

Let $(\mu_t)_{t \in [0,T]} \in \cM([0,T];\cP_2(\R^n))$ be a fixed flow of measures.
For later use, we define the Hamiltonians of the MFC problem and the MFG problem, as 
\begin{multline*}
   [0,T] \times \R^n \times \cP_2(\R^n) \times \R^n \times \R^{n \times d} \ni (t,x, \mu, p,q) \mapsto \\H(t,x,\mu,p,q)\coloneqq b(t,x)p + \operatorname{Tr}(\sigma(t,x)^\top q) + C(t,x, \mu) \in \R,
\end{multline*}
and
\begin{multline*}
    [0,T] \times \R^n \times \R^n \times \R^{n \times d} \ni (t, x,p,q) \mapsto H^{\mu}(t,x,p,q)\coloneqq b(t,x)p + \operatorname{Tr}(\sigma(t,x)^\top q) + c(t,x, \mu_t) \in \R,
\end{multline*}
respectively.
We highlight that the following relation between the two Hamiltonians holds
\begin{equation}\label{eq:MFG:Hamiltonians_relation}
    H^{\mu}(t,x,p,q) = H(t, x, \mu_t, p, q) + \int_{\R^n} \delta_\mu C(t,\tilde x, \mu_t) (x) \mu_t(d\tilde x),     
\end{equation}
for $(t, x, p,q) \in [0,T] \times \R^n \times \R^n \times \R^{n \times d}$.

\medskip
We make the following convexity assumptions:
\begin{assumption} \label{assumption:general:convexity}
For any $(t,p, q) \in [0,T] \times \R^n \times \R^{n \times d}$, the terminal cost function $(x, \mu) \mapsto G(x, \mu)$ and the map $(x,\mu) \mapsto H(t,x,\mu,p,q)$ are L-jointly convex in the sense of equation \eqref{eq:L-convexity}.
\end{assumption}
The main result of this section is as follows:
\begin{theorem}\label{thm:main}
Let Assumptions \ref{assumptions:general_part} and \ref{assumption:general:convexity} hold.
Suppose that $(\xi^*, \mu^*)$ is a solution to the potential MFG.
Then, the control $\xi^*$ is optimal for the MFC problem. 
\end{theorem}
\begin{proof}
Let $(\xi^*, \mu^*)  \in \cB \times \cM([0,T];\cP_2(\R^n))$ be a solution to the MFG, and denote by $X^*$ the solution to equation \eqref{eq:dynamics} associated with the control $\xi^*$.
Let $(p,q) = (p_t,q_t)_{t \in [0, T]}$ be the solution of the Backward Stochastic Differential Equation (BSDE)
\begin{equation} \label{eqn:MFG:adjoint}
    \begin{cases}
    dp_t = - \partial_x H^{\mu^*}\big(t,X_t^*, p_t, q_t\big) dt  + q_t dW_t, \quad t\in[0,T]\\
    p_T = \partial_x g \big(X^*_T, \mu^*_T\big),
    \end{cases}
\end{equation}
which exists and it is unique by \cite[Theorem 6.2.1]{pham2009continuous}.
Since $\xi^*$ is optimal for the cost function $J_g(\cdot,\mu^*)$, the stochastic maximum principle for singular stochastic controls (see \cite[Theorem 3.6]{bahlali2007maximum}) implies that, for any admissible strategy $\xi \in \cB$, it holds
\begin{equation}\label{eq:thm_main:negativity}
    \E\left[ \int_{0}^{T} (K(t) + {\zeta(t)}^\top p_t) \big(d\xi^{*}_t - d\xi_t\big) \right] \leq 0.
\end{equation}
Notice that \cite{bahlali2007maximum} requires the running and terminal cost to have bounded derivatives in the state variable. However, by employing the dominated convergence theorem in the usual way, the result can be extended to running and terminal cost with derivatives of at most linear growth, as given by Assumption \ref{assumptions:general_part}.

Let now $\xi \in \cB$ be an arbitrary admissible control.
Our goal is to prove $J\big(\xi^*\big) - J\big(\xi\big) \leq 0$.
By L-convexity of $G(x,\mu)$, we have
\begin{multline} \label{eqn:MFG->MFCproof:eqn1}
    J \big(\xi^*\big) - J\big(\xi\big)
    \leq \E\left[ \int_{0}^{T} \left( C\big(t,X^*_t, \mu^*_t\big) - C\big(t,X^\xi_t, \cL(X^{\xi}_t)\big) \right)dt + \int_0^T K(t)(d\xi^{*}_t - d\xi_t) \right] \\
    + \E\left[ \partial_x G\big(X^*_T, \mu^*_T\big) \big( X^*_T - X^\xi_T\big) + \tilde\E \left[ \partial_\mu G\big(X^*_T, \mu^*_T\big) (\tilde  X^*_T) \big( \tilde X^*_T - \tilde X^\xi_T\big)\right] \right],
\end{multline}
for any atomless Polish probability space $\big(\tilde{\Omega},\tilde{\cF},\tilde{\P}\big)$ and $(\tilde X^*_T, \tilde X^\xi_T)$ an independent copy of $( X^*_T,  X^\xi_T)$. We observe that
\begin{equation} \label{eqn:MFG->MFCproof:eqn2}
\begin{aligned}
    &\E\left[  \tilde\E \left[ \partial_\mu G\big(X^*_T, \mu^*_T\big) (\tilde  X^*_T) \big( \tilde X^*_T - \tilde X^\xi_T\big)\right] \right] = \E\left[  \tilde\E \left[ \partial_\mu G\big(\tilde X^*_T, \mu^*_T\big) (X^*_T) \big( X^*_T - X^\xi_T\big)\right] \right] \\
    &= \E\left[  \tilde\E \left[ \partial_\mu G\big(\tilde X^*_T, \mu^*_T\big) (X^*_T) \right] \big( X^*_T - X^\xi_T\big)\right] = \E\left[ \int_{\R^n} \partial_\mu G\big(\tilde x, \mu^*_T\big) (X^*_T)  \mu^*_T(d\tilde x) \big( X^*_T - X^\xi_T\big)\right]  \\
    &= \E\left[ \int_{\R^n} \partial_x \delta_\mu G\big(\tilde x, \mu^*_T\big) (X^*_T)  \mu^*_T(d\tilde x) \big( X^*_T - X^\xi_T\big)\right] = \E\left[ \partial_x \int_{\R^n}  \delta_\mu G\big(\tilde x, \mu^*_T\big) (X^*_T)  \mu^*_T(d\tilde x) \big( X^*_T - X^\xi_T\big)\right],
\end{aligned}
\end{equation}
where we applied Fubini's theorem in the first equality, the consistency condition \textit{(ii)} of the MFG solution in the third equality, the relation $\partial_\mu G(x,\mu)(y)=\partial_y\delta_\mu G(x,\mu)(y)$ in the fourth equality and we exchanged the derivative and the integral in virtue of the growth and differentiability assumptions in Assumption \ref{assumptions:general_part} in the last equality.
By employing the terminal condition in equation \eqref{eqn:MFG:adjoint}, equation \eqref{eqn:MFG->MFCproof:eqn2} yields
\begin{equation}
\begin{aligned}
    \E&\left[ \partial_x G\big(X^*_T, \mu^*_T\big) \big( X^*_T - X^\xi_T\big) + \tilde\E \left[ \partial_\mu G\big(X^*_T, \mu^*_T\big) (\tilde  X^*_T) \big( \tilde X^*_T - \tilde X^\xi_T\big)\right] \right]  \\
    & = \E\left[ \left( \partial_x G\big(X^*_T, \mu^*_T\big) \ + \partial_x \int_{\R^n} \delta_\mu G\big(\tilde x, \mu^*_T\big) (X^*_T)  \mu^*_T(d\tilde x) \right) \big( X^*_T - X^\xi_T\big) \right] \\
    & = \E\left[ \partial_x g(X^*_T,\mu^*_T)(X^*_T - X^\xi_T) \right]  = \E\left[ p_T \big(X^*_T- X^\xi_T\big) \right].
\end{aligned}
\end{equation}
By Itô's formula and the first equation in \eqref{eqn:MFG:adjoint}, we have
\begin{align*}
    p_T \big(X^*_T- X^\xi_T\big) &= \int_{0}^{T} \left( p_t \big(b\big(t,X^*_t\big) - b\big(t,X^\xi_t\big) \big) + \operatorname{Tr} \Big(\big(\sigma\big(t,X^*_t\big)- \sigma\big(t,X^\xi_t\big)\big)^\top q_t\Big) \right)dt \\
    &\quad - \int_{0}^{T} \partial_x H^{\mu^*}\big(t,X^*_t, p_t, q_t\big)\big(X^*_{t-}- X^\xi_{t-}\big)dt + \int_{0}^{T}  \zeta(t)^\top p_t \big(d\xi^{*}_t - d\xi_t\big) \\
    &\quad + \int_0^T \big(X^*_{t-}-X^\xi_{t-}\big)^\top q_t\,dW_t + \int_0^T p_t^\top \big(\sigma(t,X^*_t)-\sigma(t,X^\xi_t)\big)\,dW_t .
\end{align*}
We now show that the two stochastic integrals in the previous equality are true martingales. Standard estimates for equation \eqref{eq:dynamics}, together with the Lipschitz continuity of $b$ and $\sigma$, the continuity of $\zeta$, and the square-integrability of the admissible controls, yield
\[
\E\bigg[\sup_{t\in[0,T]}\big|X^*_t-X^\xi_t\big|^2\bigg]<\infty.
\]
Moreover, standard estimates for the BSDE \eqref{eqn:MFG:adjoint} yield
\[
\E\bigg[\sup_{t\in[0,T]}|p_t|^2\bigg] + \E\bigg[\int_0^T|q_t|^2dt\bigg]<\infty.
\]
Hence, by the Cauchy-Schwarz inequality,
\begin{equation*}
    \E\bigg[\bigg(\int_0^T \left|\big(X^*_{t-}-X^\xi_{t-}\big)^\top q_t\right|^2dt\bigg)^{1/2}\bigg] \leq \bigg(\E\bigg[\sup_{t\in[0,T]}\big|X^*_t-X^\xi_t\big|^2\bigg]\bigg)^{1/2}\bigg(\E\bigg[\int_0^T|q_t|^2dt\bigg]\bigg)^{1/2}<\infty.
\end{equation*}
Furthermore, if $L>0$ is a Lipschitz constant for $\sigma$, then
\begin{equation*}
    \E\bigg[\bigg(\int_0^T \left|p_t^\top\big(\sigma(t,X^*_t)-\sigma(t,X^\xi_t)\big)\right|^2dt\bigg)^{1/2}\bigg] \leq L\sqrt{T}\bigg(\E\bigg[\sup_{t\in[0,T]}|p_t|^2\bigg]\bigg)^{1/2}\bigg(\E\bigg[\sup_{t\in[0,T]}\big|X^*_t-X^\xi_t\big|^2\bigg]\bigg)^{1/2}<\infty.
\end{equation*}
Therefore, by the Burkholder-Davis-Gundy inequality, both stochastic integrals are true martingales and have zero expectation. Taking expectations in the previous equality, and noticing that $X^*_{t-}=X^*_t$ and $X^\xi_{t-}=X^\xi_t$ for $dt$-a.e. $t\in[0,T]$, we obtain
\begin{multline*}
    \E\left[ p_T \big(X^*_T- X^\xi_T\big) \right] = \E\left[ \int_{0}^{T} \left( p_t \big(b\big(t,X^*_t\big) - b\big(t,X^\xi_t\big) \big) + \operatorname{Tr} \Big(\big(\sigma\big(t,X^*_t\big)- \sigma\big(t,X^\xi_t\big)\big)^\top q_t\Big) \right)dt \right] \\
    - \E\left[ \int_{0}^{T} \partial_x H^{\mu^*}\big(t,X^*_t, p_t, q_t\big)\big(X^*_t- X^\xi_t\big)  dt \right] + \E \left[ \int_{0}^{T} \zeta(t)^\top p_t \big(d\xi^{*}_t - d\xi_t\big) \right].
\end{multline*}
Putting the last equality in \eqref{eqn:MFG->MFCproof:eqn1}, we get
\begin{multline*}
    J\big(\xi^*\big) - J\big(\xi\big) \leq  \E\left[ \int_{0}^{T} \left( H\big(t,X^*_t, \mu^*_t, p_t, q_t\big) - H\big(t,X^\xi_t, \cL(X^{\xi}_t) ,p_t, q_t\big) \right)dt \right] \\
    - \E\left[ \int_{0}^{T} \left(\partial_x H^{\mu^*}\big(t,X^*_t, p_t, q_t\big)\big(X^*_t- X^\xi_t\big)\right)dt +  \int_{0}^{T} \big( K(t) +  \zeta(t)^\top p_t \big) \big(d\xi^{*}_t - d\xi_t\big)  \right].
\end{multline*}
Now, using the hypothesis of L-jointly convexity of the map $(x,\mu) \mapsto H(t,x,\mu,p,q)$ in Assumption \ref{assumption:general:convexity}, we have
\begin{equation} \label{eqn:MFG->MFCproof:eqn3}
\begin{aligned}
    & J\big(\xi^*\big) - J\big(\xi\big) \leq   \E\left[ \int_{0}^{T} \left(  \partial_x H\big(t,X^*_t, \mu^*_t, p_t, q_t\big) \big(X^*_t- X^\xi_t\big) + \tilde\E \left[ \partial_\mu H(t,X^*_t,\mu^*_t,p_t,q_t) (\tilde X^*_t) \big(\tilde X^*_t - \tilde X^\xi_t\big)\right]\right) dt \right] \\
    &\quad \ - \E\left[ \int_{0}^{T} \partial_x H^{\mu^*}\big(t,X^*_t, p_t, q_t\big)\big(X^*_t- X^\xi_t\big)  dt - \int_{0}^{T} (K(t) + {\zeta(t)}^\top p_t) \big(d\xi^{*}_t - d\xi_t\big) \right].
\end{aligned}
\end{equation}
for any atomless Polish probability space $\big(\tilde{\Omega},\tilde{\cF},\tilde{\P}\big)$ and $(\tilde X^*, \tilde X^\xi)$ an independent copy of $( X^*,  X^\xi)$.
By noticing that $\partial_\mu H(t,x,\mu,p,q)(y) = \partial_\mu C(t,x,\mu)(y)$, the same calculations as in equation \eqref{eqn:MFG->MFCproof:eqn2} yield
\begin{equation*}
    \E\left[ \tilde\E \left[  \partial_\mu H\big(t,X^*_t, \mu^*_t, p_t, q_t \big) (\tilde  X^*_t) \big( \tilde X^*_t - \tilde X^\xi_t\big) \right] \right] = \E\left[  \partial_x \int_{\R^n} \delta_\mu C\big(t,\tilde x, \mu^*_t\big) (X^*_t)  \mu^*_t(d\tilde x) \big( X^*_t - X^\xi_t\big)\right],
\end{equation*}
and recalling equation \eqref{eqn:MFG->MFCproof:eqn3}, we obtain
\begin{equation}
\begin{aligned}
    J\big(\xi^*\big) - J\big(\xi\big) &\leq   \E\left[ \int_{0}^{T} \left(  \partial_x H\big(t,X^*_t, \mu^*_t, p_t, q_t\big) + \partial_x \int_{\R^n} \delta_\mu C\big(t,\tilde x, \mu^*_t\big) (X^*_t)  \mu^*_t(d\tilde x) \right) \big( X^*_t - X^\xi_t\big) dt\right] \\
    & \quad - \E\left[ \int_{0}^{T} \partial_x H^{\mu^*}\big(t,X^*_t, p_t, q_t\big)\big(X^*_t- X^\xi_t\big)  dt -\int_{0}^{T} (K(t) + {\zeta(t)}^\top p_t) \big(d\xi^{*}_t - d\xi_t\big) \right] \\
    &\leq \E\left[ \int_{0}^{T} (K(t) + {\zeta(t)}^\top p_t) \big(d\xi^{*}_t - d\xi_t\big)  \right] \leq 0,
\end{aligned}
\end{equation}
where the second last inequality follows from an exact cancellation after differentiating \eqref{eq:MFG:Hamiltonians_relation} in $x$, and the last inequality follows from \eqref{eq:thm_main:negativity}.
This concludes the proof.
\end{proof}

From Theorem \ref{thm:main}, we have the following simple result concerning the uniqueness of solutions to the associated potential MFG:
\begin{corollary}\label{corollary:uniqueness}
Suppose that there exists at most one solution to the MFC problem \eqref{eq:cost_MFC}. Then, there exists at most one solution to the potential MFG \eqref{eq:MFG:cost} as well.
\end{corollary}
\begin{proof}
Let $(\xi^1,\mu^1)$ and $(\xi^2,\mu^2)$ be two solutions to the potential MFG. By Theorem \ref{thm:main}, both $\xi^1$ and $\xi^2$ are solutions to the MFC problem.
By uniqueness, we get $\xi^1 = \xi^2$, which implies that $\mu^1=\mu^2$ as well, and so the two MFG solutions coincide.
\end{proof}

\begin{remark}
If the coefficients $b$ and $\sigma$ depended on the law of the state process, a connection with an associated MFG could still be formally established. More precisely, the Pontryagin system satisfied by an optimal control of the MFC problem can be identified with that of a suitably defined MFG (see \cite[Remark 5.6]{Guo-etal} in the singular-control setting and \cite[Section 6.2.5]{carmona2018probabilistic_vol1} in the regular-control setting). However, the data of this MFG, and in particular its running cost, would depend explicitly on derivatives of the value function of the original MFC problem, through the terms involving the measure derivatives of $b$ and $\sigma$. Therefore, unlike in our setting, this construction does not provide a self-contained MFG that can be solved independently to recover a solution of the MFC problem.
By contrast, when the mean-field dependence is confined to the cost functional, as in our framework, the associated MFG is self-contained, with coefficients and cost functional explicitly determined by the data of the original MFC problem. Under Assumption \ref{assumption:general:convexity}, Theorem \ref{thm:main} implies that any solution of this MFG is then a solution of the original MFC problem. This is precisely what allows us to exploit the MFG formulation to study the MFC problem, as illustrated by the monotone follower application of Section \ref{sec:application}.
\end{remark}
\section{A Case Study: A Mean-field Monotone Follower Problem}\label{sec:application}

In this section, we consider a mean-field control version of the monotone follower problem introduced by Karatzas and Shreve in \cite{karatzas1984monotone}.
Let $\rho > 0$, $\sigma >0$, $K > 0$ and $\alpha \in \R$.
The problem is as follows:
find the control $\xi^*$ which minimizes
\begin{equation}\label{eq:application:mfc_payoff}
    J(\xi) \coloneqq \E \left[ \frac{1}{2}\int_0^T e^{-\rho t}(X^\xi_t - \alpha \E[X^{\xi}_t])^2dt + K \int_{0}^{T} e^{-\rho t} d\xi_t + Ke^{-\rho T} X^{\xi}_T   \right],
\end{equation}
under the dynamics constraint
\begin{equation}\label{eq:application:dynamics}
    X^{\xi}_t = x + \sigma W_t - \xi_t,
\end{equation}
for any process $\xi=(\xi_t)_{t \in [0,T]}$ $\bbF$-adapted, right-continuous, non-decreasing, $\xi_{0-} = 0$ $\P$-a.s.\ and $\E[\xi^2_T] < \infty$.
In the following, we consider only admissible controls $\xi \in \cB$ continuous at terminal time, i.e. such that $\xi_{T-} = \xi_{T}$ $\P$-a.s.
This is without loss of generality: indeed, suppose that the control $\xi$ has a jump of magnitude $a$ at time $T$. Hence, $X^{\xi}_T = X^{\xi}_{T-} - a$, and we have 
\begin{multline*}
    K\int_{0}^{T}e^{-\rho t} d\xi_t + K e^{-\rho T}X^\xi_T = K \int_{[0,T)}e^{-\rho t} d\xi_t + K e^{-\rho T} a + Ke^{-\rho T}(X^\xi_{T-} - a) \\
    = K \int_{[0,T)}e^{-\rho t} d\xi_t + Ke^{-\rho T}X^\xi_{T-}.
\end{multline*}
Thus, the contribution of the terminal jump to the control cost is exactly canceled by its contribution to the terminal cost.
Moreover, removing the terminal jump does not affect the running cost, since it changes the state only at the single time $T$. Hence, for every $\xi \in \cB$, the control
\begin{equation}\label{eq:modified_control}
    \bar\xi_t \coloneqq \left \{\begin{aligned} & \xi_t, \quad && t < T, \\
    & \xi_{T-}, \quad && t = T,
    \end{aligned} \right.
\end{equation}
is admissible, continuous at $T$ and satisfies $J(\xi) = J(\bar\xi)$.
Therefore, the optimization problem may be equivalently restricted to controls in the class
\[
\cB_T \coloneqq \{ \xi \in \cB \, : \, \xi_{T} = \xi_{T-}, \, \P\text{-a.s.} \}.
\]
This will restore uniqueness of the optimal control (see Corollary \ref{app:cor:uniqueness}).

The MFC problem under study fits the framework of Section \ref{sec:general}, with $n=d=m=1$, $b(t,x) \equiv 0$, $\sigma(t,x) \equiv \sigma$, $\zeta(t) \equiv -1$ and 
\[
C(t,x,\mu) = \frac{e^{-\rho t}}{2}\Big( x-\alpha \int_{\R} y \mu(dy) \Big)^2, \quad G(x,\mu) = Ke^{-\rho T} x, \quad K(t) = K e^{-\rho t}.
\]
It is straightforward to see that $C(t,x,\mu)$ nd $G(x,\mu)$ satisfy Assumptions \ref{assumptions:general_part} and \ref{assumption:general:convexity}.
In particular, the linear derivative of $C(t,x,\mu)$ is given by
\[
\delta_\mu C(t,x,\mu)(y) = -\alpha e^{-\rho t}\Big(x-\alpha \int_{\R} y \mu(dy) \Big) y,
\]
so that the instantaneous cost $c(t,x,\mu)$ of the potential MFG is given by
\begin{equation*}
    c(t,x,\mu) = e^{-\rho t}\left( \frac{1}{2}\Big(x-\alpha \int_{\R} y \mu(dy) \Big)^2 - \alpha(1-\alpha) x \int_{\R} y \mu(dy) \right).
\end{equation*}
As the dependence on the measure is of scalar type, to define the potential MFG, it is enough to consider a measurable real-valued process $\theta=(\theta_t)_{t \in [0, T]}$ instead of a measurable flow of measure $\mu=(\mu_t)_{t \in [0, T]} \in \cM([0, T],\cP_2(\R))$.
Thus, for any real-valued measurable function $\theta$, the cost functional of the associated potential MFG is given by
\begin{equation}\label{eq:application:mfg_cost}
    J_g(\xi,\theta) \coloneqq \E \bigg[ \int_0^T e^{-\rho t}\bigg( \frac{1}{2}(X^\xi_t - \alpha \theta_t)^2 -\alpha(1-\alpha)\theta_t X^\xi_t \bigg) dt + K \int_{0}^{T} e^{-\rho t} d\xi_t + Ke^{-\rho T}X^{\xi}_T \bigg].
\end{equation}
We observe that, as for the cost functional $J(\cdot)$, we can restrict the search of the optima for $J_g(\cdot,\theta)$ to the class $\cB_T$, for any real-valued measurable function $\theta$.

\begin{definition}\label{app:def:MFG_sol}
We say that a pair $(\xi^*,\theta^*)$, with $\xi^* \in \cB_T$ and $\theta^*$ a measurable real-valued process, is a solution to the potential MFG if the following two properties hold:
\begin{enumerate}[label=(\roman*)]
    \item $J_g(\xi^*,\theta^*) \leq J_g(\xi,\theta^*)$ for any admissible control $\xi \in \cB_T$, and
    \item $\theta^*_t = \E[X^*_t]$ for any $t \in [0,T]$, where $X^*$ denotes the solution of \eqref{eq:application:dynamics} associated to the optimal control $\xi^*$.
\end{enumerate}
\end{definition}
By Theorem \ref{thm:main}, any solution to the potential MFG \eqref{eq:application:mfg_cost} is also a solution to the MFC problem \eqref{eq:application:mfc_payoff}.
Therefore, we now solve the potential MFG. 

\smallskip
Define the set 
\begin{equation*}
\Theta\coloneqq\Big\{ \theta: [0,T] \to \R \textit{ measurable,  càdlàg, non-increasing and bounded}  \Big\}.
\end{equation*}
Without loss of generality, we can restrict to the case where $\theta \in \Theta$. Indeed, consider $(\xi^*,\theta^*)$ as a solution to the potential MFG.
At the equilibrium, the process $\theta^*$ is $\theta^*_t = x - \E[\xi^*_t]$ for $t \in [0,T]$, which is càdlàg and non-increasing, since the optimal control $\xi^*$ is càdlàg and non-decreasing. In addition, the process $\theta^*$ is bounded in $t$. Indeed, for every $t \in [0,T]$, we have $x \geq \theta^*_t = x - \E[\xi^*_t] \geq x - \E[\xi^*_T]$, since $0 \leq \E[\xi^*_t] \leq \E[\xi^*_T] < \infty$.

\subsection{Step 1: Solving the Singular Stochastic Control Problem}\label{sec:ctrl_pb}

In this subsection we prove that, for any fixed $\theta \in \Theta$, there exists a unique optimal control $\xi^* \in \cB_T$ which minimizes $J_g(\cdot,\theta)$.
This is the content of Theorem \ref{application:teo:optimal_xi}.
To this extent, we exploit the well-known connection between singular control problems and optimal stopping problems, adapting the approach developed in \cite{baldursson1996irreversible} (see also \cite{ferrarideangleissasa}).
\smallskip

Let $(t,x)\in [0, T] \times \R$. Set $\cT_t \coloneqq \{\tau \in [0, T-t] \ \bbF\text{-stopping times} \}$, and consider the optimal stopping problem
\begin{multline}
\label{optstop:eq:optimalstopping}
    v(t,x) \coloneqq \inf_{\tau \in \cT_t} \E \bigg[ \int_0^\tau e^{-\rho s} \big(x + \sigma W_s - \alpha (2 - \alpha) \theta_{s +t} \big) ds + K e^{-\rho \tau }\ind_{\{ \tau < T-t \}} + K e^{-\rho (T-t)} \ind_{\{ \tau = T-t \}} \bigg] \\
    = \inf_{\tau \in \cT_t} \E \left[ \int_0^\tau e^{-\rho s} \big(x + \sigma W_s - \alpha (2 - \alpha) \theta_{s +t} \big) ds + Ke^{-\rho \tau } \right].
\end{multline} 
We start by proving the following simple properties of the value function $v$:
\begin{lemma} \label{optstop:lemma:v(t,x)}
\begin{enumerate}[label=(\roman*)]
    \item $v(t,x) \leq K$ for all $(t,x) \in [0, T] \times \R$. 
    \item The map $[0, T] \times \R \ni (t,x) \mapsto v(t,x) \in \R$ is continuous.
    \item For fixed $t \in [0, T]$, the map $\R \ni x \mapsto v(t,x) \in \R$ is non-decreasing.
\end{enumerate}
\end{lemma}
\begin{proof}
\emph{(i).} The upper bound follows by taking $\tau = 0$ in \eqref{optstop:eq:optimalstopping}.

\emph{(ii).} Let $(t_n, x_n)_{n \in \N} \subset [0,T] \times \R$ be a sequence converging to $(t,x) \in [0,T] \times \R$. Take $\varepsilon > 0$ and let $\tau^\varepsilon \coloneqq \tau^\varepsilon(t,x)$ be an $\varepsilon$-optimal stopping time for the optimal stopping problem with value function $v(t,x)$. Then $\tau^\varepsilon \wedge (T - t_n) \in \mathcal{T}_{t_n}$ and we have
\begin{multline*}
v(t_n,x_n) - v(t,x) \leq \varepsilon + \E \left[ \int_0^{\tau^\varepsilon \wedge (T - t_n)} e^{-\rho s} \big(x_n - x - \alpha (2 - \alpha) (\theta_{s + t_n} - \theta_{s + t}) \big) ds\right] \\
+ \E \left[ \int_{\tau^\varepsilon \wedge (T - t_n)}^{\tau^\varepsilon } e^{-\rho s} \vert x + \sigma W_s- \alpha (2 - \alpha) \theta_{s + t} \vert ds + K \big\vert e^{- \rho \tau^\varepsilon \wedge (T - t_n) } -  e^{- \rho \tau^\varepsilon} \big\vert \right] .
\end{multline*}
Notice that the last expectation on the right-hand side goes to zero as $n \to \infty$, since $0 \leq \tau^\varepsilon - \tau^\varepsilon \wedge (T - t_n) \leq \vert t_n - t \vert$.
We recall that, since $\theta \in \Theta$, it is càdlàg, non-increasing and bounded. Thus, it has at most countably many discontinuities. Therefore, for almost every $s \in [0, T-t]$ we have $\theta_{s + t_n} \to \theta_{s + t}$ as $t_n \to t$. Thus, we can apply dominated convergence to the right-hand side of the inequality above and get
\begin{equation} \label{eqn:optstop:v_cont_proof_eqn1}
    \limsup_{n \to \infty} v(t_n, x_n)\leq v(t,x) + \varepsilon.
\end{equation}
Similarly, taking $\varepsilon$-optimal stopping times $\tau^\varepsilon_n \coloneqq \tau^\varepsilon(t_n,x_n)$ for the optimal stopping problem with value function $v(t_n,x_n)$, we get $\tau^\varepsilon_n \wedge (T - t) \in \mathcal{T}_t$ and
\begin{multline*}
    v(t,x) - v(t_n,x_n) \leq \varepsilon + \E \left[ \int_0^{\tau^\varepsilon_n \wedge (T - t)} e^{-\rho s} \big(x - x_n - \alpha (2 - \alpha) (\theta_{s + t} - \theta_{s + t_n}) \big) ds \right] \\
    + \E \left[ \int_{\tau_n^\varepsilon \wedge (T - t)}^{\tau_n^\varepsilon } e^{-\rho s} \vert x_n + \sigma W_s- \alpha (2 - \alpha) \theta_{s + t_n} \vert ds + K \big\vert e^{- \rho \tau_n^\varepsilon \wedge (T - t) } -  e^{- \rho \tau_n^\varepsilon} \big\vert \right] .
\end{multline*}
Arguing as before, we can again apply dominated convergence to the right-hand side of the inequality above and get
\begin{equation} \label{eqn:optstop:v_cont_proof_eqn2}
    \liminf_{n \to \infty} v(t_n, x_n)\geq v(t,x) - \varepsilon.
\end{equation}
Equations \eqref{eqn:optstop:v_cont_proof_eqn1} and \eqref{eqn:optstop:v_cont_proof_eqn2} imply the continuity of $v$ on $[0, T] \times \R$ by arbitrariness of $\varepsilon > 0$.

\emph{(iii).} Since the term  $x + \sigma W_s$
appears linearly inside the integral in \eqref{optstop:eq:optimalstopping}, the map $x \mapsto v(t,x) $ is clearly non-decreasing.
\end{proof}

Let $\cC$ and $\cS$ be the continuation and stopping regions for the optimal stopping problem:
\begin{equation} \label{optstop:eq:CandSregions}
    \cC \coloneqq \big\{ (t,x) \in [0,T]\times \R: \ v(t,x)<K \big\}, \quad \quad \cS \coloneqq \big\{ (t,x) \in [0,T]\times \R: \ v(t,x)=K \big\}.
\end{equation}

\begin{lemma} \label{optstop:lemma:subharmonic}
Fix $(t,x) \in [0, T ] \times \R$. The process
\[
V \coloneqq \left( e^{- \rho u} v\big(t + u, x + \sigma W_u \big) + \int_0^{u} e^{-\rho s} \big(x + \sigma W_s - \alpha (2 - \alpha) \theta_{s + t} \big) ds\right)_{u \in [0, T - t]}
\]
is an $\bbF$-submartingale and it holds
\begin{equation} \label{optsopt:eq:v_subharmonic}
    v(t,x) \leq \E \left[ e^{- \rho \tau} v\big(t + \tau, x + \sigma W_\tau\big) + \int_0^{\tau} e^{-\rho s} \big(x + \sigma W_s - \alpha (2 - \alpha) \theta_{s +t} \big) ds \right],
\end{equation}
for every $\tau \in \cT_t$. Moreover, the stopping time
\begin{equation} \label{optstop:eq:peskirtau}
    \tau^* = \tau^*(t,x) \coloneqq \inf \big\{ s \in [0, T-t]: \ \big(t+s, x + \sigma W_s) \in \cS \big\} \wedge (T - t), 
\end{equation}
is optimal for problem \eqref{optstop:eq:optimalstopping} and the process $(V_{u \wedge \tau^*})_{u \in [0, T - t]}$ is an $\bbF$-martingale.
\end{lemma}
\begin{proof}
The submartingale property of the process $V$ is straightforward from $(ii)$ of Lemma \ref{optstop:lemma:v(t,x)}, in particular from the upper-semicontinuity of $v(t, x)$, and from \cite[Theorem 2.4]{peskir2006optimal}.
Therefore, \eqref{optsopt:eq:v_subharmonic} is true for any $\tau \in \cT_t$.
Exploiting again the upper-semicontinuity of $v$ and \cite[Corollary 2.9]{peskir2006optimal}, we have that $\tau^*$ as in \eqref{optstop:eq:peskirtau} is optimal. Since \eqref{optsopt:eq:v_subharmonic} holds with equality for $\tau^*$ as in \eqref{optstop:eq:peskirtau}, the martingale property of $(V_{u \wedge \tau^*})_{u \in [0, T - t]}$ follows.
\end{proof}

By exploiting the non-decreasing property of the map $x \mapsto v(t,x)$ for fixed $t \in [0,T]$, we can define the free boundary between $\cC$ and $\cS$ by
\begin{equation} \label{optstop:eqn:bounday_b}
    b_t \coloneqq \inf\big\{ x\in \R: \ v(t,x)\geq K\big\},
\end{equation}
with the convention $\inf \emptyset = + \infty$. By employing $(b_t)_{t \in [0,T]}$, \eqref{optstop:eq:CandSregions} can be equivalently written as
\begin{equation} 
    \cC = \big\{ (t,x) \in [0,T]\times \R: \ x < b_t \big\}, \quad \quad \cS \coloneqq \big\{ (t,x) \in [0,T]\times \R: \ x \geq b_t \big\}.
\end{equation}
In addition, we can rewrite the optimal stopping time \eqref{optstop:eq:peskirtau} as
\begin{equation} \label{optstop:eqn:tau_optimal}
\tau^*(t,x) = \inf \big\{ s \in [0, T-t]: \, x + \sigma W_s \geq b_{s+t}\big\} \wedge (T - t).
\end{equation}
The next lemma summarizes some useful properties of $b=(b_t)_{t \in [0,T]}$:
\begin{lemma}\label{optstop:lemma:properties_b}
\begin{enumerate}[label=(\roman*)]
    \item For any $t \in [0,T)$, $b_t \geq \alpha(2-\alpha)\theta_t + \rho K$.
    \item For any $t \in [0,T)$, $b_t < \infty$, while $b_T = -\infty$.
\end{enumerate}
\end{lemma}
\begin{proof}
$(i)$. Recall \eqref{optsopt:eq:v_subharmonic}.
Take $(t,x) \in \cS$ with $t < T$, consider $u \in (0, T-t)$ and the stopping time $\tau \wedge u < T - t$, for any $\tau \in \cT_t$, $\tau >0$ $\P$-a.s.
Since $(t,x) \in \cS$ implies $v(t,x)=K$, we rewrite \eqref{optsopt:eq:v_subharmonic} as
\begin{equation*}
\begin{aligned}
    K &\leq \E \left[ e^{- \rho (\tau \wedge u)} v\big(t + \tau \wedge u, x + \sigma W_{  \tau \wedge u }\big) + \int_0^{ \tau \wedge u} e^{-\rho s} \big(x + \sigma W_s - \alpha (2 - \alpha) \theta_{s +t} \big) ds \right] \\
    &\leq \E \left[ Ke^{- \rho (\tau \wedge u)} + \int_0^{ \tau \wedge u} e^{-\rho s} \big(x + \sigma W_s - \alpha (2 - \alpha) \theta_{s +t} \big) ds \right] \\
    &= \E \left[ \int_0^{\tau \wedge u} e^{-\rho s} \big(x + \sigma W_s - K \rho - \alpha (2 - \alpha) \theta_{s +t} \big) ds + K \right],  
\end{aligned}
\end{equation*}
where the second inequality follows from $(i)$ of Lemma \ref{optstop:lemma:v(t,x)}. Thus, we obtain
\[
0 \leq \lim_{u \to 0} \frac{1}{u} \E \left[ \int_0^{\tau \wedge u} e^{-\rho s} \big(x + \sigma W_s - K \rho - \alpha (2 - \alpha) \theta_{s +t} \big) ds\right] = x - K \rho - \alpha (2 - \alpha) \theta_{t}
\]
for any $(t,x) \in \cS$, $t < T$. This implies the inclusion
\begin{equation} \label{eqn:lemma:uperbound_xi_star_proof_eqn2}
    \cS \cap [0,T) \times \R \subseteq \big\{ (t,x) \in [0,T)\times \R: \ x \geq K \rho + \alpha (2 - \alpha) \theta_{t} \big\}.
\end{equation}
Taking the complementary set in \eqref{eqn:lemma:uperbound_xi_star_proof_eqn2} relative to $[0,T) \times \R$ and rewriting the continuation region $\cC$ with respect to the free boundary function $b = (b_t)_{t \in [0, T]}$, we get
\[
\big\{ (t,x) \in [0,T) \times \R: \ x < K \rho + \alpha (2 - \alpha) \theta_{t} \big\} \subseteq  \big\{ (t,x) \in [0,T) \times \R: \ x < b_t \big\},
\]
which implies $b_t \geq  K \rho  + \alpha (2 - \alpha) \theta_{t}$ for $t \in [0, T)$, and concludes the proof.

$(ii)$. The fact that $b_T = - \infty$ follows from the definition of $v(t,x)$ and $b_t$.
To prove that $b_t < \infty$ for any $t \in [0,T)$, first notice that $\alpha(2-\alpha)\theta_{t+s} \leq \alpha(2-\alpha)\kappa$, with $\kappa = \theta_0$ if $\alpha \in (0,2)$ and $\kappa = \theta_T$ if $\alpha \in \R \setminus (0,2)$.
Then, we have
\begin{align*}
    v(t,x) & \geq \inf_{\tau \in \cT_t} \E \left[ \int_0^\tau e^{-\rho s} \big(x + \sigma W_s - \alpha (2 - \alpha) \kappa \big) ds + Ke^{-\rho \tau } \right] \\
    & \geq \inf_{\tau \geq 0} \E \left[ \int_0^\tau e^{-\rho s} \big(x + \sigma W_s - \alpha (2 - \alpha) \kappa \big) ds + Ke^{-\rho \tau } \right] =: v_{\kappa}(x).
\end{align*}
Then, it is enough to notice that, since $v(t,x) \geq v_{\kappa}(x)$, we have $\cS_t\coloneqq \{x \in \R: \, v(t,x) \geq K\} \supseteq \{x \in \R: \, v_{\kappa}(x) \geq K\}$, which implies $b_t \leq \underline{b} \coloneqq \inf\{x \in \R: \, v_{\kappa}(x) \geq K\}$. A standard exponential-martingale estimate yields $ v_{\kappa}(x) = K$ for $x \geq \alpha (2 - \alpha) \kappa + \rho K + \tfrac{\sigma}{\sqrt{2 \rho}}$ and therefore $\underline{b} <\infty$. This concludes the proof.
\end{proof}

Fix $(t,x) \in [0, T ] \times \R$. We introduce the non-decreasing process
\begin{equation} \label{optstop:eqn:runngin_sup_xi}
    \xi^*_u \coloneqq \left\{ \begin{aligned}
        \sup_{s \in [0,u]} \left( x + \sigma W_s - b_{s + t}\right)^+, && \quad u \in[0,T-t),\\
        \xi^*_{(T-t)-}, && \quad u = T -t, 
    \end{aligned} \right.
\end{equation}
with $(b_t)_{t\in[0,T]}$ as in \eqref{optstop:eqn:bounday_b}.

\begin{proposition}
The process $\xi^*$ of \eqref{optstop:eqn:runngin_sup_xi} is an admissible control, and it holds
\begin{equation}\label{eq:estimate:xi_star}
    \xi^*_u \leq \sup_{s \in [0,u]} (x + \sigma W_s - K\rho - \alpha(2-\alpha)\theta_{s+t})^+,  \quad u \ \in [0, T-t].
\end{equation}
\end{proposition}
\begin{proof}
First of all, $\xi^*_u$ is a.s. finite, for any $u \in [0,T-t)$, since $-\infty <b_{s+t} < \infty$ by Lemma \ref{optstop:lemma:properties_b}.
Furthermore, $\xi^*$ is $\bbF$-adapted since the boundary function is deterministic. 
To prove that $\xi^*$ is admissible, we first show that $u \mapsto \xi^*_u$ is right-continuous with left-limits on $[0,T-t)$. Clearly, $u \mapsto \xi^*_u$ admits left-limits since it is non-decreasing.
In particular, this implies that $\xi^*_{(T-t)-}$ exists and so the process in \eqref{optstop:eqn:runngin_sup_xi} is well-defined.
To show that $\xi^*$ has right-continuous paths, we first notice that $t \mapsto x-b_{t}$ is upper-semicontinuous. Indeed, one has
\[
\big\{ (t,x) \in [0,T]\times \R: \ v(t,x)<K \big\} = \big\{ (t,x) \in [0,T]\times \R: \ x < b_t \big\}.
\]
The set on the left-hand side above is open since it is the preimage of an open set via the upper-semicontinuous mapping $(t, x)\mapsto v(t,x)$ (cf. $(ii)$ of Lemma \ref{optstop:lemma:v(t,x)}). Hence, the set on the right-hand side is open as well, and thus $(t,x) \mapsto x-b_t$ is upper-semicontinuous. In particular, for fixed $t \in [0, T)$, the map  $s \mapsto x-b_{s + t}$ is upper-semicontinuous. Therefore, since the sum of an upper-semicontinuous and a continuous function is upper-semicontinuous, we have
\[
\limsup_{u\downarrow s} \big( x + \sigma W_u- b_{u + t} \big) \leq x + \sigma W_s -b_{s + t}. 
\]
Moreover, we obtain
\begin{multline} \label{optstop:eqn:xi_admissible_proof}
    \lim_{u\downarrow s} \xi^*_u = \xi^*_s \vee  \lim_{u\downarrow s} \sup_{r \in (s, u]}  \left( x + \sigma W_r - b_{r + t}\right)^+ = \xi^*_s \vee  \limsup_{u\downarrow s}  \left( x + \sigma W_u - b_{u + t}\right)^+ \\ \leq  \xi^*_s \vee \left( x + \sigma W_s - b_{s + t}\right)^+ = \xi^*_s.
\end{multline}
Since $\lim_{u\downarrow s} \xi^*_u \geq \xi^*_s$ by monotonicity of $u \mapsto \xi^*_u$, \eqref{optstop:eqn:xi_admissible_proof} implies right-continuity.
Furthermore, \eqref{eq:estimate:xi_star} follows from $(i)$ of Lemma \ref{optstop:lemma:properties_b} for any $t \in [0,T)$. Therefore, by taking the limit as $u \uparrow T -t$, we get
\begin{multline*}
    \xi^*_{T-t} = \xi^*_{(T-t)-} \leq \lim_{u \uparrow T -t } \sup_{ s \in [0,u] } (x + \sigma W_s -K\rho - \alpha(2-\alpha)\theta_{t+s})^+ \\
    = \sup_{ s \in [0,T - t) } (x + \sigma W_s -K\rho - \alpha(2-\alpha)\theta_{t+s})^+  \leq \sup_{ s \in [0,T-t] } (x + \sigma W_s -K\rho - \alpha(2-\alpha)\theta_{t+s})^+,
\end{multline*}
which proves that the bound holds also at terminal time.
Finally, the condition $\E \big[(\xi^*_{T -t })^2\big] < \infty$ follows from \eqref{eq:estimate:xi_star} since $\theta$ is bounded. Indeed, we have
\begin{multline*}
    \E\left[\left(\xi^*_{T -t }\right)^2\right] = \E\left[\left(\xi^*_{(T-t)-} \right)^2\right] \leq \E \bigg[  \bigg( \sup_{s \in [0,T-t]} \big(x + \sigma W_s - K\rho - \alpha(2-\alpha)\theta_{s + t}\big)^+ \bigg)^2\bigg] \\ \leq \E \bigg[ \sup_{s \in [0,T-t]} \Big\vert x + \sigma W_s - K\rho - \alpha(2-\alpha)\theta_{s + t}\Big\vert^2 \bigg] \leq 2 \sigma^2 \E \bigg[  \sup_{s \in [0,T-t]}  \vert W_s \vert^2 \bigg] + C < \infty,
\end{multline*}
where $C>0$ is a positive constant.
\end{proof}

The main result of this subsection is as follows:
\begin{theorem} \label{application:teo:optimal_xi}
For any fixed $\theta \in \Theta$, the control 
\begin{equation}\label{eq:optimal_control}
    \xi^*_t = \left\{ \begin{aligned}
        \sup_{s \in [0,t]} \left(x + \sigma W_s -b_s \right)^+, \quad && t \in [0, T), \\
        \xi^*_{T-}, \quad && t = T,
    \end{aligned} \right.
\end{equation}
is a solution of the singular control problem associated with cost functional $J_g(\cdot,\theta)$, and it is unique in the class $\cB_T$.
\end{theorem}
\begin{proof}
To see uniqueness, it is enough to notice that the cost functional $J_g(\cdot,\theta)$ is strictly convex on $\cB_T$ for any fixed $\theta \in \Theta$.
Indeed, take $\lambda \in (0,1)$, $\xi^1$ and $\xi^2$ in $\cB_T$ and set $\xi^\lambda=\lambda\xi^1 + (1-\lambda)\xi^2 \in \cB_T$.
Then, it holds $X^{\xi^\lambda}_t = \lambda X^{\xi^1}_t + (1-\lambda)X^{\xi^2}_t$ for any $t \in [0,T]$ $\P$-a.s.
Combining this observation with the strict convexity of the square function and the linearity of the integral with respect to the control variable, it holds $J_g(\xi^\lambda,\theta) < \lambda J_g(\xi^1,\theta) + (1-\lambda) J_g(\xi^2,\theta)$.

We now deal with the optimality. We borrow arguments from \cite{baldursson1996irreversible} (see also \cite{ferrarideangleissasa}). Fix $(t,x) \in [0,T]\times \R$ and recall $v(t,x)$ as defined in \eqref{optstop:eq:optimalstopping}. Define the functions
\begin{equation} \label{optstop:eq:functionsPhiUphi}
\begin{aligned} 
     &\Phi(t,x) \coloneqq \E \bigg[ \int_0^{T-t} e^{-\rho s} \left( \frac{1}{2}\big({x + \sigma W_s} - \alpha \theta_{s +t} \big)^2 - \alpha (1 - \alpha) \theta_{s +t} {(x + \sigma W_s)}\right)ds + K e^{-\rho (T-t)} (x+\sigma W_{T - t} )\bigg], \\
     &\varphi(t,x) \coloneqq \partial_x \Phi(t,x) = \E \left[ \int_0^{T-t} e^{-\rho s} \big({x + \sigma W_s} - \alpha (2 - \alpha) \theta_{s +t} \big) ds + K e^{-\rho (T-t)}  \right], \\
     &U(t,x) \coloneqq \Phi(t,x) + \int_{-\infty}^{x} \left( v(t,y) - \varphi(t,y)\right)dy.
\end{aligned}
\end{equation}
We show that $U(0,x)$ is the value function for the singular control problem associated with the cost functional \eqref{eq:application:mfg_cost}, for fixed $\theta \in \Theta$, and $\xi^*$ in \eqref{optstop:eqn:runngin_sup_xi} is the optimal control.
Take an admissible control $(\xi_t)_{t\in [0, T]} \in \cB_{T}$ and define its right-continuous inverse (cf. \cite[Chapter 0, Section 4]{revuz2013continuous}) as
\begin{equation} \label{optstop:eqn:proof_theoB_eqn0}
    \tau^\xi(z) \coloneqq \inf \{ s \in [0, T-t]: \xi_s > z \} \wedge (T - t), \quad z\geq0.
\end{equation}
The process $\tau^\xi:= \{\tau^\xi(z), z \geq 0 \}$ has increasing, right-continuous sample paths and hence it admits left-limits
\begin{equation} \label{optstop:eqn:proof_theoB_eqn00}
    \tau_-^\xi(z) \coloneqq \inf \{ s \in [0, T-t]: \xi_s \geq z \} \wedge (T - t), \quad z \geq 0. 
\end{equation}
The set of points $z \in \R^+$ at which $\tau^\xi(z)(\omega)\neq \tau_-^\xi(z)(\omega)$ is countable $\P$-a.s.
Since $\xi$ is right-continuous and $\tau^\xi(z)$ is the first entry time of an open set, it is an $(\cF_{t+})_{t \in [0,T]}$-stopping time for any given $z \geq 0$. Since $(\cF_t)_{t \in [0,T]}$ is right-continuous, $\tau^\xi(z)$ is an $(\cF_t)_{t \in [0,T]}$-stopping time.
Moreover, $\tau_-^\xi(z)$ is the first entry time of the right-continuous process $\xi$ into a closed set, and hence it is an $(\cF_t)$-stopping time as well for any given $z\geq0$.
By Lemma \ref{optstop:lemma:subharmonic}, we have
\begin{equation}\label{optstop:eqn:proof_theoB_eqn000}
    v(t,x) \leq 
    \E \left[ e^{- \rho \tau^\xi(z)} v\left(t + \tau^\xi(z), x + \sigma W_{ \tau^\xi(z)}\right) + \int_0^{\tau^\xi(z)} e^{-\rho s} \big({x + \sigma W_s} - \alpha (2 - \alpha) \theta_{s +t} \big) ds \right], 
\end{equation} 
for any $z \geq 0$ and $(t,x) \in [0, T] \times \R$. Then, for any $(t,x) \in [0, T] \times \R$, taking $z = x - y$, $x \geq y$, by \eqref{optstop:eq:functionsPhiUphi} we have
\begin{multline*}
    U(t,x) - \Phi(t,x) \leq \int_{-\infty}^{x} \E \bigg[ e^{- \rho \tau^\xi(x-y)} v\left(t +  \tau^\xi(x-y), y + \sigma W_{ \tau^\xi(x-y)}\right) \\
    + \int_0^{\tau^\xi(x-y)} e^{-\rho s} \big(y + \sigma W_s - \alpha (2 - \alpha) \theta_{s +t} \big) ds  \\
    - \int_0^{T-t} e^{-\rho s} \big(y + \sigma W_s - \alpha (2 - \alpha) \theta_{s +t} \big) ds - K e^{-\rho (T-t) } \bigg] dy.
\end{multline*}
By $(i)$ of Lemma \ref{optstop:lemma:v(t,x)} we get 
\begin{equation} \label{optstop:eqn:proof_theoB_eqn1}
\begin{aligned}
    U(t,x) - \Phi(t,x) &\leq \int_{-\infty}^{x} \E \bigg[ K e^{- \rho \tau^\xi(x-y)} - K e^{-\rho (T -t)} -\int_{\tau^\xi(x-y)}^{T-t} e^{-\rho s} \big(y + \sigma W_s - \alpha (2 - \alpha) \theta_{s +t} \big) ds \bigg] dy \\
    &= \int_{-\infty}^{x} \E \bigg[ K \left( e^{- \rho \tau^\xi(x-y)} -  e^{- \rho (T -t)}\right)\ind_{\{\tau^\xi(x-y)< T -t\} } \\
    &\quad \ - \int_{0}^{T-t} \ind_{\{s > \tau^\xi(x-y)\}} e^{-\rho s} \big(y + \sigma W_s - \alpha (2 - \alpha) \theta_{s +t} \big) ds \bigg] dy,
\end{aligned}
\end{equation}
Applying Fubini's theorem in the last expression of \eqref{optstop:eqn:proof_theoB_eqn1}, we obtain
\begin{multline}\label{optstop:eqn:proof_theoB_eqn2}
    U(t,x) - \Phi(t,x) \leq \E \left[ \int_{-\infty}^{x} \Big(K \left( e^{- \rho \tau^\xi(x-y)} -  e^{- \rho (T -t)} \right) \ind_{\{\tau^\xi(x-y)< T -t\} } \right. \\
    \left.-\int_{0}^{T-t} \ind_{\{s > \tau^\xi(x-y)\}} e^{-\rho s} \big(y + \sigma W_s - \alpha (2 - \alpha) \theta_{s +t} \big) ds \Big) dy \right].
\end{multline}  
The change of variable formula of \cite[Proposition 4.9, Chapter 0, p.8]{revuz2013continuous} (see also \cite[Equation (4.7)]{baldursson1996irreversible}) implies
\begin{equation} \label{optstop:eqn:proof_theoB_eqn3}
    \int_{-\infty}^{x} \left( e^{- \rho \tau^\xi(x-y)} - e^{- \rho (T -t)} \right) \ind_{\{\tau^\xi(x-y)< T -t\} }  dy = \int_0^{T-t} \left( e^{- \rho s} - e^{- \rho (T -t)} \right) d\xi_s.
\end{equation}
Moreover, $s > \tau^\xi(x-y)$ if and only if $y > x -\xi_s$, $s \geq 0$, and therefore, from \eqref{optstop:eqn:proof_theoB_eqn1} and from the change of variable formula of \cite[Proposition 4.9, Chapter 0, p.8]{revuz2013continuous} (see also \cite[Equation (4.7)]{baldursson1996irreversible}), we get
\begin{multline*}
    U(t,x) - \Phi(t,x) \leq \E \left[ K\int_{0}^{T-t} e^{- \rho s} d\xi_s - K e^{-\rho (T - t)} \xi_{T -t} \right.\\
    \left. - \int_{0}^{T-t} e^{-\rho s} \left( \int_{x -\xi_s}^{x} \big(y + \sigma W_s - \alpha (2 - \alpha) \theta_{s +t} \big) dy \right) ds\right].
\end{multline*} 
Now, solving the integral with respect to the variable $y$ and adding and subtracting the term $K e^{- \rho (T - t)} ( x + \sigma W_{T -t} )$, we get
\begin{multline}
\label{optstop:eqn:proof_theoB_eqn4}
    U(t,x) - \Phi(t,x) \leq \E \bigg[ \int_0^{T-t} e^{-\rho s} \left( \frac{1}{2}\big(X^{\xi}_s - \alpha \theta_{s +t} \big)^2 - \alpha (1 - \alpha ) \theta_{s + t} X^{\xi}_s\right)ds  \\ + K \int_{0}^{T-t} e^{- \rho s} d\xi_s + K e^{- \rho (T - t) } X^{\xi}_{T -t} \bigg] - \E \bigg[ K e^{- \rho (T - t) } (x + \sigma W_{T -t}) \\
    + \int_0^{T-t} e^{-\rho s} \left( \frac{1}{2}\big(x + \sigma W_s - \alpha \theta_{s +t} \big)^2 - \alpha (1 - \alpha ) \theta_{s +t} (x + \sigma W_s)\right)ds \bigg].
\end{multline} 
Since $t \in [0,T]$, $\xi \in \cB_T$ are arbitrary, taking $t=0$ we get $U(0,x) \leq \inf_{\xi \in {  \mathcal{B}_T}} J_g(\xi,\theta)$.

We now show that picking $\xi^*$ as in \eqref{optstop:eqn:runngin_sup_xi} in the arguments above, all the inequalities become equalities, due to \eqref{optstop:eqn:tau_optimal}. Fix $z\in \R^+$, take $u \in { [0, T-t)}$ arbitrary.
Note that, by \eqref{optstop:eqn:proof_theoB_eqn00} and \eqref{optstop:eqn:tau_optimal}, we have $\P$-a.s. the equivalences
\begin{multline*}
    \tau^{\xi^*}_-(z) \leq u \Longleftrightarrow \xi^*_u \geq z \Longleftrightarrow  \sup_{s \in [0,u]} \left( x + \sigma W_s - b_{s + t}\right)^+ \geq z \\
    \Longleftrightarrow x + \sigma W_r - b_{r + t} \geq z,  \text{ for some } r \in[0, u] \Longleftrightarrow \tau^*(t, x - z) \leq u.
\end{multline*}
Therefore we can conclude that $\tau^{\xi^*}_-(z) = \tau^*(t, x - z)$ $\P$-a.s. and for a.e. $z \geq 0$. By \eqref{optstop:eqn:proof_theoB_eqn0} and \eqref{optstop:eqn:proof_theoB_eqn00}, we also have $\tau^{\xi^*}_-(z) = \tau^{\xi^*}(z)$ $\P$-a.s. and for a.e. $z \geq 0$; hence $\tau^{\xi^*}(z) = \tau^*(t, x - z)$ $\P$-a.s. and for a.e. $z \geq 0$. Now, take $\xi = \xi^*$ to obtain equality in \eqref{optstop:eqn:proof_theoB_eqn000}, by { Lemma \ref{optstop:lemma:subharmonic}}. Optimality of $\tau^* = \tau^{\xi^*}$ also gives equality in \eqref{optstop:eqn:proof_theoB_eqn1}; then, we can interchange the integrals and argue as from \eqref{optstop:eqn:proof_theoB_eqn1} to \eqref{optstop:eqn:proof_theoB_eqn4} to obtain $U(0,x) = J_g(\xi^*,\theta)$, which implies $U(0,x) = \inf_{\xi \in \cB_T} J_g(\xi, \theta)$ and $\xi^*$ is optimal.
\end{proof}

\subsection{Step 2: The Fixed-Point Problem and the Mean-field Equilibrium}\label{sec:fixed_point}

In this subsection, we solve the potential MFG.
Given the findings of Section \ref{sec:ctrl_pb}, we just need to show that there exists $\theta^* \in \Theta$ that satisfies the consistency condition $\theta^*_t = \E[X^*_t]$, $t \in [0,T]$. To this extent, we show that there exists a well-defined map $\Psi$ from a subset of the space $\Theta$ into itself such that $\theta^*$ is given by the fixed-point of such map. This is the content of Theorem \ref{application:theo:fixed-point}.

\smallskip
In the following, when needed, we reinforce the notation of the deterministic free boundary function $b=(b_t)_{t \in [0, T]}$ and the optimal process $\xi^*=(\xi^*_t)_{t \in [0,T]}$, by making explicit the dependence on the measurable process $\theta = (\theta_t)_{t \in [0, T]}$.

\begin{lemma}\label{lemma:monotonicity_b_xi}
Let the process $\xi^*(\theta)$ be as given by Theorem \ref{application:teo:optimal_xi}. If $\alpha \in (0,2)$, then the maps $\theta \mapsto b(\theta)$ and $\theta \mapsto \xi^*(\theta)$ are non-decreasing and non-increasing, respectively, that is
\[
\theta_t \leq \theta_t' \;\; dt\text{-a.e.} \implies b_t(\theta) \leq b_t(\theta')  \;\; \forall t \in [0, T],  \textit{ and } \xi^*_t(\theta) \geq \xi^*_t(\theta') \ \P\text{-a.s.}, \, \forall t \in [0, T].
\]
Conversely, if $\alpha \in \R \backslash (0,2)$, then the maps $\theta \mapsto b(\theta)$ and $\theta \mapsto \xi^*(\theta)$ are non-increasing and non-decreasing, respectively.
\end{lemma}
\begin{proof}
Take $t \in [0, T]$. Recall from \eqref{optstop:eq:optimalstopping} and \eqref{optstop:eqn:bounday_b} the definition of $v(t,x;\theta)$ and $b_t(\theta)$, where we now stress the dependence with respect to $\theta$ in the notation.
Let $\alpha \in (0,2)$. Then, the term $\alpha (2 - \alpha)$ is always positive. Therefore, taking $\theta$ and $\theta'$ such that $\theta_{s+t} \leq \theta_{s+t}'$ $ds$-a.e., it is easy to see that we have $v(t,x;\theta) \geq v(t,x;\theta')$, and therefore 
\[
b_t(\theta) = \inf\big\{ x\in \R: \ v(t,x;\theta)\geq K\big\} \leq \inf\big\{ x\in \R: \ v(t,x;\theta')\geq K\big\} = b_t(\theta'),
\]
which proves that the map $\theta \mapsto b(\theta)$ is non-decreasing.
As a consequence, since the optimal control takes the form of running supremum (cf. Theorem \ref{application:teo:optimal_xi}), the map $\theta \mapsto \xi^*(\theta)$ is non-increasing.
Conversely, if we take $\alpha \in \R \backslash (0,2)$, then the term $\alpha (2 - \alpha)$ is always non-positive, and, by the same reasoning as above, $\theta \mapsto b(\theta)$ is non-increasing and $\theta \mapsto \xi^*(\theta)$ is non-decreasing.
\end{proof}

{  
\begin{lemma}\label{lemma:delta_bounds}
There exists a positive constant $\delta$ large enough, depending on $\sigma$, $T$, $x$, $\alpha$, $K$ and $\rho$, such that
\begin{enumerate}[label = (\roman*)]
    \item If $\alpha \in \R\setminus (0,2)$ and $\theta \equiv x $, we have $\E [\xi^*_T(x)] \leq \delta$.
    \item If $\alpha \in (0,2) \setminus \{1\}$ and $\theta \equiv x - \delta$, we have $\E [\xi^*_T(x-\delta)] \leq \delta$.
\end{enumerate}
Finally, if $\alpha = 1$, set $\underline{\theta}_t = x - \delta_t$, where $\delta_t = \rho K + \frac{\sigma^2}{2\rho K}t$. Then, $\E [\xi^*_t(\underline{\theta})] \leq \delta_t$ for any $t \in [0,T]$. 
\end{lemma}
\begin{proof}
Consider the case when $\alpha \in \mathbb{R} \setminus (0, 2)$. By \eqref{eq:estimate:xi_star} we get
\begin{equation}
\begin{aligned}
    \E \left[\xi^*_T(x)\right] &\leq \E \left[\sup_{ 0 \leq s \leq T }\big( x + \sigma W_s -K\rho -\alpha (2-\alpha)x\big)^+\right] = \E \bigg[\bigg( \left(1 - \alpha (2-\alpha)\right)x - K\rho + \sigma \sup_{ 0 \leq s \leq T } W_s\bigg)^+\bigg] \\
    &\leq \big( \left(1 - \alpha (2-\alpha)\right)x - K\rho\big)^+ + \sigma \E \left[ \sup_{ 0 \leq s \leq T } W_s \right]  = \big( \left(1 - \alpha (2-\alpha)\right)x - K\rho\big)^+ + \sigma \sqrt{\frac{2T}{\pi}} \leq \delta.
\end{aligned}
\end{equation}
The first equality follows since the map $\mathbb{R} \ni y \mapsto y^+ \in \mathbb{R}^+$ is non-decreasing; for the second equality, see \cite[Problem 82, p. 96]{karatzas1991brownian}.
Now, consider the case when $\alpha \in (0, 2) \setminus \{1\}$. By \eqref{eq:estimate:xi_star} and the same computations as above, we get
\begin{equation}
\begin{aligned}
    \E \left[\xi^*_T(x - \delta)\right] &
    \leq \big( \left(1 - \alpha (2-\alpha)\right)x - K\rho+ \alpha (2 - \alpha) \delta\big)^+  + \sigma \sqrt{\frac{2T}{\pi}} \\
    &\leq \alpha (2 - \alpha) \delta + \big( \left(1 - \alpha (2-\alpha)\right)x - K\rho\big)^+  + \sigma \sqrt{\frac{2T}{\pi}} \leq \delta,
\end{aligned}
\end{equation}
where the last inequality follows from $\alpha (2 - \alpha) \in (0, 1)$ and taking $\delta$ large enough.
Finally, suppose that $\alpha=1$. By \eqref{eq:estimate:xi_star}, for every $t\in[0,T]$ we have
\begin{equation}
    \E[\xi_t^*(\underline{\theta})]\leq\E\left[\sup_{0\leq s\leq t}\left(x+\sigma W_s-\rho K-\underline{\theta}_s\right)^+\right]=\E\left[\sup_{0\leq s\leq t}\left(\sigma W_s+\frac{\sigma^2}{2\rho K}s\right)\right].
\end{equation}
Here the positive part can be omitted since the process inside the supremum is equal to zero at $s=0$. Moreover, a direct calculation using the explicit distribution of the running supremum of a Brownian motion with drift (see \cite[1.1.4 p.250]{Borodin}) gives
\begin{equation} \label{stima}
    \E\left[\sup_{0\leq s\leq t}\left(\sigma W_s + \lambda s\right)\right]\leq\lambda t+\frac{\sigma^2}{2\lambda},\qquad \lambda>0.
\end{equation}
Finally, using the estimate \eqref{stima} with $\lambda=\frac{\sigma^2}{2\rho K}$, we obtain $\E[\xi_t^*(\underline\theta)]\leq\frac{\sigma^2}{2\rho K}t+\rho K=\delta_t$.
\end{proof}
}

Let $\delta$ be a constant depending on $\sigma$, $T$, $x$, $\alpha$, $K$ and $\rho$, as in Lemma \ref{lemma:delta_bounds}. 
{ Set
\begin{equation}
\underline{\theta}_t = \left\{ \begin{aligned}
    x-\delta, && \quad \alpha \neq 1, \\
    x - \rho K - \frac{\sigma^2}{2\rho K}t, && \quad \alpha = 1
\end{aligned} \right.
\end{equation}}
and define the set 
\begin{equation}
\cE\coloneqq\Big\{ \theta \in  \Theta \textit{ such that }  { \underline{\theta}_t} \leq \theta_t \leq x, \textit{ for } t \in [0, T] \Big\}.
\end{equation}
For $\theta \in \cE$, consider the map $\Psi$ defined as 
\begin{equation} \label{eq:map_psi}
    \Psi(\theta) \coloneqq \left(\E \left[ X^{*}_t(\theta)\right]\right)_{t\in [0,T]},
\end{equation}
where $X^{*}(\theta)$ denotes the solution of \eqref{eq:application:dynamics} associated to the optimal control $\xi^*(\theta)$.
It is clear that, if $\theta^*$ is a fixed point of $\Psi$, the pair $(\xi^*(\theta^*),\theta^*)$ is a solution to the potential MFG.

\begin{theorem} \label{application:theo:fixed-point}
The map $\Psi:\cE \to \cE$ is well defined and admits a fixed point.
\end{theorem}
\begin{proof}
To show the existence of a fixed point, we apply Kakutani-Fan-Glicksberg fixed-point theorem to $\Psi$.
{ The proof is divided into several steps.}

\smallskip
{ \textbf{Step 1}. We first prove that $\Psi$ maps $\cE$ into itself.}
Since $\Psi_t(\theta) = x - \E [\xi^*_t(\theta)]$, it is càdlàg and non-increasing. The upper bound follows directly from the definition of the map $\Psi(\theta)$ and from the positivity of $\E [\xi^*_t(\theta)]$.
To prove the lower bound, we treat the case $\alpha \in (0,2) { \setminus\{1\}}$, the case $\alpha \in \R \backslash (0,2)$, { and the case $\alpha = 1$} separately.
First, take $\alpha \in (0,2){ \setminus\{1\}}$. Since $\theta \in \cE$, we have $\theta \geq x - \delta$, and, by Lemma \ref{lemma:monotonicity_b_xi} and Lemma \ref{lemma:delta_bounds}, we get
\begin{equation}
    \Psi_t(\theta) =  x - \E \left[\xi^*_t(\theta)\right] \geq x - \E \left[\xi^*_T(\theta)\right] \geq x - \E \left[\xi^*_T(x-\delta)\right] \geq x - \delta, \quad t\in[0,T].
\end{equation}
Now, take $\alpha \in \R \backslash (0,2)$. Since $\theta \in \cE$, we have $\theta \leq x$, and, by Lemma \ref{lemma:monotonicity_b_xi} and Lemma \ref{lemma:delta_bounds}, we also get the lower bound, as
\begin{equation}
    \Psi_t(\theta) =  x - \E \left[\xi^*_t(\theta)\right] \geq x - \E \left[\xi^*_T(\theta)\right] \geq x - \E \left[\xi^*_T(x)\right] \geq x - \delta, \quad t\in[0,T].
\end{equation}
{ Finally, take $\alpha = 1$. Since $\theta \in \cE$, we have $\theta_t \geq \underline{\theta}_t$ for every $t \in [0, T]$, and, by Lemma \ref{lemma:monotonicity_b_xi} and Lemma \ref{lemma:delta_bounds}, we get
\begin{equation}
    \Psi_t(\theta) =  x - \E \left[\xi^*_t(\theta)\right] \geq x - \E \left[\xi^*_t(\underline{\theta})\right] \geq x - \delta_t = \underline{\theta}_t, \quad t\in[0,T].
\end{equation}}

{ \textbf{Step 2}. We show that the set $\cE$ is compact and convex in a suitable topology.}
Consider $\cE$ as a subset of $L^2([0, T])$ endowed with the topology of weak convergence of functions, and notice that $\cE$ is convex and weakly compact in $L^2([0, T])$. Indeed, if a sequence $\theta^n$ converges to $\theta$ weakly in $L^2([0, T])$, it is always possible to find a non-increasing càdlàg version of $\theta$ (see, e.g., \cite[Lemmata 4.5 and 4.6]{karatzas1984monotone}) so that the bounds ${ \underline{\theta}_t} \leq \theta_t \leq x$ still hold.
Since $\cE$ is obviously convex and closed for the strong topology, it is closed for the weak topology and, by relying again on the bounds ${ \underline{\theta}_t}  \leq \theta_t \leq x$, it is also norm-bounded.
{ Since $L^2([0,T])$ is reflexive, by the Banach-Alaoglu theorem its closed bounded subsets are weakly compact. Therefore, $\cE$ is compact in the weak topology of $L^2([0,T])$.}

\smallskip
{ In the next steps, we verify} that the map $\Psi:\cE \to \cE$ has closed graph.
To this extent, let $(\theta^n)_{n \geq 1} \subset \cE$, $\theta \in \cE$ and $y \in \cE$ such that $\theta^n \to \theta$ and $\Psi(\theta^n) \to y$ weakly in $L^2([0,T])$.
We prove that $y = \Psi(\theta)$.

\smallskip
{ \textbf{Step 3}. We reduce the proof of the closed-graph property of $\Psi$ to the case of a.e. and strong $L^2$ convergence of the sequence $(\theta^n)_{n\geq 1}$.}
Since the functions $(\theta^n)_{n \geq 1}$ are uniformly bounded and non-increasing, Helly's selection theorem { (see, e.g., \cite[Chapter 7, Exercise 13]{rudin1976principles})} implies that there exists a (relabeled) subsequence $\theta^{n}$ and $\bar{\theta} \in \cE$ so that $\theta^n_t \to \bar{\theta}_t$ for $dt$-a.e. $t \in [0, T]$.
It is easy to see that $\bar{\theta} = \theta$ $dt$-a.e. Indeed, since $\theta^n \to \bar{\theta}$ a.e. and $\theta^n$ and $\bar{\theta}$ are uniformly bounded, we have that $\theta^n$ converges to $\bar{\theta}$ strongly in $L^2([0, T])$ as well, which implies, by uniqueness of the weak limit, that $\bar{\theta} = \theta$.
Thus, from now on, we suppose that $(\theta^n)_{n \geq 1}$ converges to $\theta$ both a.e. and in $L^2$.

\smallskip
{ \textbf{Step 4}. We show that, along a subsequence, the optimal controls associated to $\theta^n$ converge weakly to some $\bar{\xi}$ and that $y_t=x-\E[\bar{\xi}_t]$ for $dt$-a.e. $t\in[0,T]$.}
Let $(\xi^*(\theta^n))_{n \geq 1}$ be the sequence of optimal controls associated to $\theta_n$.
By employing estimate \eqref{eq:estimate:xi_star} on $\xi^*(\theta^n)$, it follows that 
\begin{multline}\label{eq:estimate:measure_mu}
    \E\left[\int_0^T\vert \xi^*_t(\theta^n) \vert^2 dt + \vert \xi^*_T(\theta^n) \vert^2 \right] \leq (T+1) \E[\vert \xi^*_T(\theta^n) \vert^2 ] \\
    \leq 2(T+1)\E\left[ \sup_{s \in [0,T]}(x - K\rho + \sigma W_s)^2 \right] + 2\alpha^2(\alpha-2)^2(T+1)\sup_{ s \in [0,T]} (\theta^n_s)^2 \leq C, 
\end{multline}
where $C$ is a positive constant independent of $n$, since $\theta^n$ are uniformly bounded by definition of $\cE$.

To ease the notation, from now on, we denote $\xi^*(\theta^{n})$ simply by $\xi^n$.
Define the measure $\mu$ on $[0,T]$ by setting $\mu(dt) = dt + \delta_{T}(dt)$.
Consider the space $L^2([0,T]\times\Omega, \cP, \mu \otimes \P)$, where $\cP$ denotes the progressive $\sigma$-algebra on $[0,T]\times\Omega$.
We denote $L^2([0,T]\times\Omega, \cP, \mu \otimes \P)$ simply by $L^2([0,T]\times\Omega)$.
By \eqref{eq:estimate:measure_mu}, the sequence $(\xi^n)_{n \geq 1}$ is bounded in norm in $L^2([0,T] \times \Omega)$.
Thus, there exists a subsequence $(\xi^{n_k})_{k \geq 1}$ and $\bar{\xi} \in L^2([0,T] \times \Omega)$ so that $\xi^{n_k}$ converges weakly (in the Hilbert sense) to $\bar{\xi}$. By employing again \cite[Lemmata 4.5 and 4.6]{karatzas1984monotone}, $\bar{\xi}$ can be taken $\bbF$-adapted, non-decreasing and right-continuous.
Let now $\phi=(\phi_u)_{u \in [0,T]} \in L^2([0,T])$ be a test-function.
Since $\xi^{n_k}$ converges weakly to $\bar{\xi}$ and $\Psi(\theta^{n_k})$ converges weakly to $y$, we get
\begin{multline*}
    \int_0^T \phi_u y_u du = \lim_{k \to \infty} \int_0^T \phi_u \Psi_u(\theta^{n_k}) du = \int_0^T x \phi_u du - \lim_{k \to \infty} \int_0^T \phi_u \E[\xi_u^{n_k}] du \\
    = \int_0^T x \phi_u du - \lim_{k \to \infty} \E\left[ \int_0^T \phi_u\xi_u^{n_k} du \right] = \int_0^T \phi_u (x - \E[\bar{\xi}_u])du,
\end{multline*}
where in the second to last equality we have used $\phi \in L^2([0,T]) \subseteq L^2([0,T]\times\Omega)$ and the weak convergence of $\xi^{n_k}$ to $\bar{\xi}$.
Since the equality above holds for any $\phi \in L^2([0,T])$, we conclude that $y_t = x - \E[\bar{\xi}_t]$ for $dt$-a.e. $t \in [0,T]$.

\smallskip
{  
\textbf{Step 5.} We start by proving that the weak limit $\bar{\xi}$ satisfies $J_g(\bar{\xi},\theta)\leq J_g(\xi,\theta)$, for every $\xi \in \cB_T$.
Then, replacing $\bar{\xi}$, if necessary, by its continuous modification at time $T$ as in \eqref{eq:modified_control}, and keeping the same notation, we may assume that $\bar{\xi}\in\cB_T$. Since this modification does not change the value of $J_g(\cdot,\theta)$, $\bar{\xi}$ is optimal for the control problem associated with $\theta$, and hence, by uniqueness, $\bar{\xi}=\xi^*(\theta)$.
}
This implies that $y_t=x-\E[\xi_t^*(\theta)]=\Psi_t(\theta)$ $dt$-a.e., thus concluding the proof.
As $\xi^n$ is the unique optimal control for the cost functional $J_g(\cdot,\theta^n)$, it holds
\begin{equation} \label{eq:estimate:optimality_xi_n}
    J_g(\xi^n,\theta^n) \leq J_g(\xi,\theta^n),
\end{equation}
for any admissible control $\xi \in \cB_T$.
Recall that the sequence $(\theta^{n_k},\xi^{n_k})_{k \geq 1}$ converges weakly to $(\theta,\bar{\xi})$. 
By Banach-Saks theorem, there exists a further (relabeled) subsequence $(\theta^{n_k},\xi^{n_k})_{k \geq 1}$ so that its Cesàro means converge strongly to $(\theta,\bar{\xi})$, i.e.
\[
(\bar{\theta}^j,\bar{\xi}^j) \coloneqq \frac{1}{j}\sum_{k=1}^j (\theta^{n_k},\xi^{n_k}) \to (\theta,\bar{\xi})
\]
in the norm sense in $L^2([0,T]) \times L^2([0,T] \times \Omega)$.
Therefore, there exists a further (relabeled) subsequence $(\bar{\theta}^{j},{ \bar{\xi}^{j}})_{j \geq 1}$ so that $\bar{\theta}^j \to \theta$ $dt$-a.e. and $\bar{\xi}^{j} \to \bar{\xi}$ $\mu \otimes \P$-a.e.
We now choose a constant $c$, such that the function $(x,\theta) \mapsto \frac{1}2{}(x - \alpha \theta)^2 - \alpha (1-\alpha)\theta  x  + c\theta^2$ is jointly convex in $(x,\theta)$.
By adding the term $\E \big[ \int_0^T e^{-\rho t} c (\theta^{n_k}_t)^2 dt\big]$ in both sides of \eqref{eq:estimate:optimality_xi_n}, we have
\begin{multline} \label{eq:estimate:optimality_xi_n2}
    \E \Bigg[ \int_0^T e^{-\rho t}\left( \frac{1}{2}(X^{\xi^{n_k}}_t - \alpha \theta^{n_k}_t)^2 -\alpha(1-\alpha)\theta^{n_k}_t X^{\xi^{n_k}}_t + c \left(\theta^{n_k}_t\right)^2\right) dt + K \int_{0}^{T} e^{-\rho t} d\xi^{n_k}_t  { + K e^{- \rho T} X^{\xi^{n_k}}_T }\Bigg] \\ \leq  J_g(\xi,\theta^{n_k}) + \E \left[ \int_0^T e^{-\rho t} c \left(\theta^{n_k}_t\right)^2 dt\right].
\end{multline}
The term on the right-hand side converges to $J_g(\xi,\theta) + \E \left[ \int_0^T e^{-\rho t} c \theta_t^2 dt\right]$ as $k \to \infty$ by the dominated convergence theorem, recalling that the sequence $(\theta^{n_k})_{k \geq 1}$ is uniformly bounded pointwise on $[0, T]$. We now take the average for $k=1, \dots, j$ on both sides of 
\eqref{eq:estimate:optimality_xi_n2} and, { by convexity of the map $(x,\theta) \mapsto \frac{1}2{}(x - \alpha \theta)^2 - \alpha (1-\alpha)\theta  x  + c\theta^2$, by linearity of $\xi \mapsto X^\xi$, and by Jensen's inequality in its discrete form}, we get
\begin{multline} \label{eq:estimate:optimality_xi_n3}
    \E \Bigg[ \int_0^T e^{-\rho t}\left( \frac{1}{2}\big(X^{\bar{\xi}^j}_t - \alpha \bar{\theta}^j_t\big)^2 -\alpha(1-\alpha)\bar{\theta}^j_t X^{\bar{\xi}^j}_t + c \big(\bar{\theta}^j_t\big)^2\right) dt + K \int_{0}^{T} e^{-\rho t} d\bar{\xi}^j_t  { + K e^{- \rho T} X^{\bar{\xi}^{j}}_T }\Bigg] \\
    \leq \frac{1}{j}\sum_{k=1}^j \left( J_g(\xi,\theta^{n_k}) + \E \left[ \int_0^T e^{-\rho t} c \left(\theta^{n_k}_t\right)^2 dt\right] \right).
\end{multline}

\smallskip
{ \textbf{Step 6}. We conclude the proof of the optimality of $\bar{\xi}$ by passing to the limit in \eqref{eq:estimate:optimality_xi_n3}.}
The term on the right-hand side still converges to $J_g(\xi,\theta) + \E \left[ \int_0^T e^{-\rho t} c \theta_t^2 dt\right]$ as $j \to \infty$ because it is the average of a convergent series. Now, we focus on the left-hand side of \eqref{eq:estimate:optimality_xi_n3}. Recalling that $L^2([0,T]\times\Omega) = L^2([0,T]\times\Omega,\mu\otimes\P)$ and $\bar{\xi}^{j} \to \bar{\xi}$ $\mu \otimes \P$-a.e., integration by parts yields
\begin{multline*}
    \lim_{j \to \infty} \E \left[\int_{0}^{T}e^{-\rho t}d\bar{\xi}^{j}_t \right] = \lim_{j \to \infty} \E \left[ e^{-\rho T}\bar{\xi}^{j}_T + \rho\int_0^T\bar{\xi}^{j}_t e^{-\rho t} dt\right] \\
    = \E \left[ e^{-\rho T}\bar{\xi}_T + \rho\int_0^T\bar{\xi}_t e^{-\rho t} dt\right] = \E\left[ \int_{0}^{T}e^{-\rho t}d\bar{\xi}_t \right].
\end{multline*}
By Vitali's theorem, recalling that the sequence $\big(\bar{\theta}^j, \bar{\xi}^j_t\big)_{j\geq 1}$ is strongly convergent in $L^2([0, T]) \times L^2([0,T] \times \Omega)$, which, { together with the linearity of $\xi\mapsto X^\xi$, implies uniform integrability of the quadratic and mixed terms}, we get
\begin{multline*}
    \lim_{j \to \infty} \E \left[ \int_0^T e^{-\rho t}\left( \frac{1}{2}\big(X^{\bar{\xi}^j}_t - \alpha \bar{\theta}^j_t\big)^2 -\alpha(1-\alpha)\bar{\theta}^j_t X^{\bar{\xi}^j}_t + c \big(\bar{\theta}^j_t\big)^2\right) dt\right] \\
    = \E \left[ \int_0^T e^{-\rho t}\left( \frac{1}{2}(X^{\bar{\xi}}_t - \alpha \theta_t)^2 -\alpha(1-\alpha)\theta_t X^{\bar{\xi}}_t + c \theta_t^2\right) dt  \right].
\end{multline*}
{ Finally, by linearity of $\xi \mapsto X^{\xi}$, we have $Ke^{-\rho T}X^{\bar{\xi}^j}_T \to Ke^{-\rho T} X^{\bar{\xi}}_T$. }
Therefore, taking the limit for $j \to \infty$ in 
\eqref{eq:estimate:optimality_xi_n3}, we obtain
\begin{multline*}
    \E \Bigg[ \int_0^T e^{-\rho t}\left( \frac{1}{2}(X^{\bar{\xi}}_t - \alpha \theta_t)^2 -\alpha(1-\alpha)\theta_t X^{\bar{\xi}}_t + c \theta_t^2\right) dt + K \int_{0}^{T}e^{-\rho t}d\bar{\xi}_t { + K e^{- \rho T} X^{\bar{\xi}}_T } \Bigg] \\ 
    \leq J_g(\xi,\theta) + \E \left[ \int_0^T e^{-\rho t} c \theta_t^2 dt\right],
\end{multline*}
which concludes the proof after the subtraction of the term $\E \big[ \int_0^T e^{-\rho t} c \theta_t^2 dt\big]$ from both sides of the inequality.
\end{proof}

\begin{corollary}\label{app:cor:uniqueness}
The solution to the potential MFG is unique. 
\end{corollary}
\begin{proof}
It is enough to show that there exists at most one solution to the mean-field singular control problem { in the class $\cB_T$}.
{ If $\alpha \neq 1$}, the map $\xi \mapsto J(\xi)$ is strictly convex { on $\cB_T$}.
This follows from the linearity of the dynamics and of the expectation, together with the strict convexity of the square and the linearity of the integral with respect to the control variable. 
{ If $\alpha = 1$, strict convexity fails. Nevertheless, also in this case, $J(\cdot)$ has at most one minimizer in $\cB_T$. Indeed, since $\xi \in \mathcal{B}_T$, by integration by parts we can write the cost functional as
\begin{equation} \label{app:cor:uniqueness_proof:eqn1}
    J(\xi) = K e^{- \rho T} x + \frac{1}{2}\int_0^T e^{-\rho t} \text{Var} (X^{\xi}_t) dt + K \rho \int_0^Te^{-\rho t} \E [\xi_t]dt. 
\end{equation}
Now, suppose that $\xi_1, \xi_2 \in \mathcal{B}_T$ are both optimal. Take the convex combination $\xi^\lambda \coloneqq \lambda \xi^1 + (1 - \lambda) \xi^2$, with $\lambda \in (0,1)$, which is again an element of $\mathcal{B}_T$. Since the control enters linearly in the state equation, from \eqref{app:cor:uniqueness_proof:eqn1} we obtain
\begin{equation} \label{app:cor:uniqueness_proof:eqn2}
    J(\xi^\lambda) = \lambda J(\xi^1) + (1 - \lambda) J(\xi^2) - \frac{\lambda (1 -\lambda)}{2}\int_0^T e^{-\rho t} \text{Var} (\xi^1_t - \xi^2_t) dt. 
\end{equation}
Because $\lambda (1- \lambda) >0$, $\xi^1$ and $\xi^2$ are both optimal and $\xi^\lambda$ is admissible, the integral on the right-hand side of \eqref{app:cor:uniqueness_proof:eqn2} must be zero. Hence $\text{Var} (\xi^1_t - \xi^2_t) = 0$ for a.e. $t \in [0, T]$, and, therefore, $\xi^1$ and $\xi^2$ may only differ by a deterministic time-dependent shift. So, for a.e. $t \in [0, T]$, there exists a deterministic $d_t$ such that $\xi^1_t = \xi^2_t + d_t$ a.s. At this point, we can take the minimum $\hat{\xi}_t \coloneqq \xi^1_t \wedge \xi^2_t = \xi^1_t - d^+_t$ for $t \in [0,T]$, which is an element of $\mathcal{B}_T$. Since $d^+_t$ is deterministic, by \eqref{app:cor:uniqueness_proof:eqn1} and the optimality of $\xi^1$, we obtain
\begin{equation} \label{app:cor:uniqueness_proof:eqn3}
    0 \leq J(\hat{\xi}) - J(\xi^1)  = K \rho \int_0^Te^{-\rho t} \left( \E [\hat{\xi}_t] - \E [\xi^1_t]  \right)dt = - K \rho \int_0^Te^{-\rho t} d^+_tdt.
\end{equation}
Since the right-hand side of \eqref{app:cor:uniqueness_proof:eqn3} cannot be positive, the integral must be zero. Hence $d_t \leq 0$ for a.e. $t \in [0,T]$. We can repeat the same argument as above exchanging the roles of $\xi^1$ and $\xi^2$ to conclude that $d_t \geq 0$ for a.e. $t \in [0, T]$, and thus, $d_t = 0$ for a.e. $t \in [0, T]$. Hence $\xi^1 = \xi^2$ $dt \otimes d\mathbb{P}$-a.e. and, since they are both càdlàg and continuous at $T$, they are indistinguishable.}
Thus, by Corollary \ref{corollary:uniqueness}, { together with the equivalence between optimization over $\cB$ and $\cB_T$,} the solution to the potential MFG is unique.
\end{proof}


We have thus proved the following result.
\begin{theorem}
The unique solution to the Mean-field Monotone Follower Problem \eqref{eq:application:mfc_payoff} is given by
\begin{equation}\label{eq:optimalreflMFC}
\xi^*_t = \left\{ \begin{aligned}
    \sup_{s \in [0,t]} \left(x + \sigma W_s - b_s(\theta^*) \right)^+, \quad && t \in [0, T), \\
    { \xi^*_{T-}, }\quad && t = T.
\end{aligned} \right.
\end{equation}
\end{theorem}
\begin{proof}
By Theorem \ref{application:theo:fixed-point}, the potential MFG of Definition \ref{app:def:MFG_sol} admits a solution $(\xi^*,\theta^*)$, which is unique by Corollary \ref{app:cor:uniqueness}.
Thus, Theorem \ref{thm:main} ensures $\xi^*$ as in  \eqref{eq:optimalreflMFC} is the unique optimal control for the MFC problem \eqref{eq:application:mfc_payoff} { in the class $\cB_T$}.
\end{proof}

\subsection{Free Boundary Analysis and Characterization of the Equilibrium for $\alpha \in (0,2)$}\label{sec:free_boundary}

In this section, we provide an explicit characterization of the solution to the potential MFG problem associated with the cost functional \eqref{eq:application:mfg_cost}. This will be shown to be triggered by a moving free boundary, depending on time and the mean-field parameter $(\theta_t)_{t\in [0,T]}$.

We restrict our analysis to the case $\alpha \in (0,2)$.
In this case, the value function of the optimal stopping problem \eqref{optstop:eq:optimalstopping} is monotone in time, which implies the monotonicity of the free boundary. This, in turn, yields regularity of the optimal stopping value function (such as the validity of the so-called smooth-fit property) and, as a final result, allows us to characterize the free boundary, { continuously extended at the terminal time,} as the unique continuous solution to a suitable nonlinear integral equation, { for each fixed continuous mean-field parameter}.

\smallskip

\begin{lemma} \label{freeb:lemma}
Let $\alpha \in (0,2)$ and { $\theta \in \cE$}. 
We have:
\begin{enumerate}[label=(\roman*)]
    \item For fixed $x \in \R$, the map $[0, T] \ni t \mapsto v(t,x) \in \R$ is non-decreasing.
    \item The map $t \mapsto b_t$ is non-increasing in $[0, T]$ and right-continuous in { $[0, T)$}.
    { In particular, there exists $b_{T-} \coloneqq \lim_{t \uparrow T} b_t \in \R$.}
    \item For fixed $t \in [0, T]$, the map $\R \ni x \mapsto v(t,x) \in \R$ is concave.
    \item The smooth-fit property holds at the free boundary $b$, that is 
    \[
    \partial_x v(t, b_t-) = \partial_x v(t, b_t+) = 0 \quad \text{for any } t \in [0, T).
    \]

    \item { $\partial_x v, \partial_{xx} v \in C(\cC)$ and the weak derivative $\partial_t v$ belongs to $L^\infty_{loc}(\cC)$. Moreover,}
    \begin{equation}\label{eq:PDE}
    \frac{1}{2}\sigma^2\partial_{xx}v(t,x)-\rho v(t,x)+\partial_t v(t,x)+x-\alpha(2-\alpha)\theta_t=0
    \end{equation}
    { for a.e. $(t,x)\in\cC$, while $v(t,x)=K$ for every $(t,x)\in\cS$.}

    \item { If $\theta \in \cE$ is continuous, $v \in C^{1, 2} \big(\cC) \cap  C^{\infty} (\mathring \cS\big)$ and \eqref{eq:PDE} holds for any $(t,x) \in \cC$ in the classical sense.}
\end{enumerate}
\end{lemma}
\begin{proof}
$(i).$ Since the map $t \mapsto \theta_t$ is non-increasing and $ \alpha (2 - \alpha) > 0$, the map $[0, T] \ni t \mapsto v(t,x) \in \R$ is non-decreasing.

$(ii).$ Take $t_1, t_2 \in [0, T]$ such that $t_1 <t_2$. Since we already proved in point $(i)$ that the map $[0, T] \ni t \mapsto v(t,x) \in \R$ is non-decreasing, we can conclude that $b_{t_1} \geq b_{t_2}$.
As for the right-continuity of the map $[0, T) \ni t \mapsto b_t \in \R$, consider a sequence $(t_n)_{\ n \in \N} $ such that $t_n \downarrow t$ as $n \to \infty$. Since $b$ is non-increasing, $b_{t+} \coloneqq \lim_{n \to \infty} b_{t_n} \leq b_t$.
Furthermore, since $(t_n,b_{t_n}) \in \cS$ for every $n \in \N$ and $\cS$ is closed, the limit $(t, b_{t+})$ is again inside $\cS$, and, therefore it holds $b_{t+} \geq b_t$.
Thus, we get $b_{t+} = b_t$ and we conclude.
{ Finally, the limit $b_{T-} \coloneqq \lim_{t \uparrow T} b_t$ exists since $t \mapsto b_t$ is non-increasing. For any $t_0 < T$, it holds $b_{T-} \leq b_{t_0} < \infty$. On the other hand, $b_t \geq \alpha(2-\alpha)\theta_t + \rho K$ by $(i)$ of Lemma \ref{optstop:lemma:properties_b}, and therefore, letting $t \uparrow T$, $b_{T-} \geq \alpha(2-\alpha)\theta_{T-} + \rho K> - \infty$.}


$(iii).$ Take $t \in [0, T]$. The concavity of the map $\R \ni x \mapsto v(t,x) \in \R$ follows from the fact that the pointwise infimum of affine functions is concave.

$(iv).$ Take $t \in { [0, T)}$. We notice that the map $x \mapsto v(t, x)$ admits both left and right derivatives at every point of its domain since it is concave by point $(iii)$.
We start by proving $\partial_x v(t, b_t + ) =  0$. By the definition of the continuation region, we have
\[
\frac{v(t, b_t + \varepsilon) - v(t, b_t)}{\varepsilon} = \frac{K - K}{\varepsilon} = 0,
\]
for any $\varepsilon > 0$. Thus, taking the limit as $\varepsilon \to 0$, we can conclude.
We now focus on left derivatives. Taking $\varepsilon > 0 $, we have $b_t - \varepsilon \in \cC_t \coloneqq \{x \in \R: \ x < b_t\}$ and
\[
\frac{v(t, b_t) - v(t, b_t - \varepsilon)}{\varepsilon} \geq \frac{K - K}{\varepsilon} = 0.
\]
Therefore, taking the limit as $\varepsilon \to 0$, we get $\partial_x v(t, b_t - )  \geq  0$.
Conversely, take $\varepsilon > 0$ and let $\tau_\varepsilon \coloneqq \tau(t, b_t - \varepsilon)$ be the optimal stopping time for the optimal stopping problem with value function $v(t, b_t - \varepsilon)$. Then
\begin{equation} \label{eqn:freeb:lemma_proof_eqn2}
    \frac{v(t, b_t) - v(t, b_t - \varepsilon)}{\varepsilon} \leq \frac{1}{\varepsilon} \E \left[ \int_0^{\tau_\varepsilon} e^{- \rho s} \left( b_t - (b_t - \varepsilon) \right)ds\right] = \E \left[ \int_0^{\tau_\varepsilon} e^{- \rho s} ds\right].
\end{equation}
We now claim that $\tau^\varepsilon$ goes to zero as $\varepsilon \to 0$. Then, by taking the limit in \eqref{eqn:freeb:lemma_proof_eqn2}, we get $\partial_x v(t, b_t -)  \leq  0$ and thus $\partial_x v(t, b_t -)  =  0$.
To complete the proof, we only need to show that
\[
\lim_{\varepsilon \to 0} \tau_\varepsilon = \lim_{\varepsilon \to 0}  \inf\big\{ s\in[0,T-t]: \ b_t - \varepsilon + \sigma W_s \geq b_{s + t}\big\} = 0, \quad a.s.
\]
Since $(\tau_\varepsilon)_\varepsilon$ is non-decreasing, there exists $\tau_0$ such that $\tau_\varepsilon(\omega) \to \tau_0(\omega)$ a.e. $\omega \in \Omega$. Assume that there exists $\Omega_0 \subset \Omega$ such that $\P(\Omega_0) > 0$ and $\tau_0 (\omega) > 0$ for any $\omega \in \Omega_0$. Take $\bar \omega \in \Omega_0$. Then, there exists a positive constant $\bar \delta = \delta(\bar \omega) > 0$ such that $\tau_0 (\bar \omega) > \bar \delta$ and $\sigma W_s < b_{s + t} - b_t + \varepsilon$ for any $s \in \big[0, \big(\tau_0 - \bar \delta /2\big)\big]$. Therefore, taking the limit as $\varepsilon \to 0$, by $(ii)$, we obtain $W_s(\bar \omega) \leq (b_{s + t} - b_t)/\sigma \leq 0, \quad \forall s \in \left[0, \left(\tau_0 - \bar \delta /2\right)\right]$,
which implies that $\P(\Omega_0) = 0$ by the law of iterated logarithm. Thus, we have a contradiction.

$(v).$ Since the value function is identically equal to $K$ in $\cS$, it obviously holds $v \in C^{\infty}(\mathring \cS)$. Thus, we can focus on the continuation region. Let $(t, x) \in \cC$. For $t_1 < t < t_2$ and $x_1 < x < x_2$, consider a rectangle $\cR \coloneqq (t_1, t_2) \times (x_1, x_2)$ such that its closure $\overline{\cR} \subset \cC$.
{ Notice that necessarily $t_2<T$, since $\cC\cap(\{T\}\times\R)=\emptyset$.}
Define the parabolic boundary $\partial_P \cR$ by the horizontal lines $[t_1, t_2) \times \{x_i\}, \ i = 1, 2$ and by the vertical line ${t_2} \times [x_1, x_2]$, and consider the following Cauchy-Dirichlet problem
\begin{equation} \label{eqn:freeb:lemma_proof_eqn3}
\begin{cases}
    \frac{1}{2} \sigma^2 \partial_{xx} u(t,x) - \rho u(t,x) + \partial_t u(t,x) + x - \alpha (2 - \alpha) \theta_t = 0, &\quad \text{in } \cR\\
    u = v, &\quad \text{on } \partial_P \cR.
\end{cases}
\end{equation}
Define $ a(t) \coloneqq e^{\rho t} \int_{t_1}^t e^{-\rho s} \allowbreak \alpha (2 - \alpha) \theta_s ds$, and consider the Cauchy-Dirichlet problem
\begin{equation}
\label{eqn:freeb:lemma_proof_eqn3bis}
\begin{cases}
    \frac{1}{2} \sigma^2 \partial_{xx} \tilde u(t,x) - \rho \tilde u(t,x) + \partial_t \tilde u(t,x) + x = 0, &\quad \text{in } \cR\\
    \tilde u = v - a, &\quad \text{on } \partial_P \cR. 
\end{cases}
\end{equation}
Since $a(t)$ is { continuous}, $v$ is continuous by Lemma \ref{optstop:lemma:v(t,x)} and the source term is $C^\infty$, by \cite[Theorem 10.3] {san_paolo_baldi} there exists a unique $\tilde u \in C^{1,2}(\mathcal{R}) \cap C(\mathcal{\overline R})$ which solves \eqref{eqn:freeb:lemma_proof_eqn3bis}.
{ Since $a\in \bbW^{1,\infty}((t_1,t_2))$ and $a'(t)-\rho a(t)=\alpha(2-\alpha)\theta_t$ for a.e. $t\in(t_1,t_2)$, the function $u(t,x) \coloneqq \tilde{u}(t,x)+a(t)$ has weak time derivative in $L^\infty_{loc}(\cR)$, satisfies $\partial_xu,\partial_{xx}u\in C(\cR)$, and solves \eqref{eqn:freeb:lemma_proof_eqn3} for a.e. $(t,x)\in\cR$.}
It remains to show that $u$ coincides with $v$ in $\cR$. Define the stopping time $\tau_{\cR} \coloneqq \inf \big\{ s \geq 0: (t + s, { x + \sigma W_s}) \in \partial_P \cR \big\}$. 
{  
Applying Itô's formula to $e^{- \rho s} \tilde{u}\big(t + s, { x + \sigma W_s}\big)$ between $0$ and $\tau_{\cR}$, and taking expectation, we get
\begin{equation} \label{eqn:freeb:lemma_proof_eqn4}
\begin{aligned}
    \tilde{u}(t, x) &= \E \left[ e^{- \rho \tau_{\cR}} \tilde{u}\big(t + \tau_{\cR}, x + \sigma W_{\tau_{\cR}} ) - \int_0^{\tau_{\cR}} e^{- \rho s} \left( \partial_t + \frac{1}{2} \sigma^2 \partial_{xx} - \rho \right) \tilde{u}(t + s, x + \sigma W_s) ds\right] \\
    &= \E \left[ e^{- \rho \tau_{\cR}} \left( v\big(t + \tau_{\cR}, x + \sigma W_{\tau_{\cR}} \big) - a(t+\tau_{\cR}) \right)+ \int_0^{\tau_{\cR}} e^{- \rho s} (x + \sigma W_s) ds\right] \\
    & = \E \left[ e^{- \rho \tau_{\cR}} v\big(t + \tau_{\cR}, x + \sigma W_{\tau_{\cR}} \big) - e^{\rho t} \int_{t_1}^t e^{-\rho s}\alpha(2-\alpha)\theta_s ds \right. \\
    & \quad \left. - e^{\rho t}\int_{t}^{t + \tau_{\cR}}e^{-\rho u} \alpha(2-\alpha)\theta_u du  + \int_0^{\tau_{\cR}} e^{- \rho s} (x + \sigma W_s) ds\right] \\
    & = \E \left[ e^{- \rho \tau_{\cR}} v\big(t + \tau_{\cR}, x + \sigma W_{\tau_{\cR}} \big)+  \int_0^{\tau_{\cR}} e^{- \rho s} \left( x + \sigma W_s - \alpha(2-\alpha)\theta_{t+s} \right)  ds\right] - a(t),
\end{aligned}
\end{equation}
where the second equality follows from \eqref{eqn:freeb:lemma_proof_eqn3bis} and the third equality from the definition of $a(t)$.}
Recall the optimal stopping $\tau^*$ and the process $V$ defined in Lemma \ref{optstop:lemma:subharmonic}. Since $\cR \subseteq \cC$, we have $\tau_{\cR} \leq \tau^* \ \P$-a.s., so that the term inside the expected value in the right-hand side of \eqref{eqn:freeb:lemma_proof_eqn4} is indeed the process $V$ at $\tau_{\cR}$. Therefore, we obtain
\begin{equation*}
\begin{aligned}
    u(t, x) &= \tilde{u}(t,x) + a(t)  \\
    &= \E \left[ e^{- \rho ({\tau_{\cR} \wedge \tau^*})} v\big(t + {\tau_{\cR} \wedge \tau^*}, x + \sigma W_{\tau_{\cR} \wedge \tau^*} \big) + \int_0^{{\tau_{\cR} \wedge \tau^*}} e^{- \rho s} \left( x + \sigma W_s - \alpha (2 - \alpha) \theta_{t + s} \right) ds\right] = v(t, x),
\end{aligned}
\end{equation*}
where the third equality follows by optional sampling theorem, since the process $(V_{u \wedge \tau^*})_{u \in [0, T -t]}$ is an $\bbF$-martingale by Lemma \ref{optstop:lemma:subharmonic}. By arbitrariness of $(t, x) \in \cC$, we conclude.

{ $(vi).$ It is enough to notice that, if $\theta \in \cE$ is continuous, $a$ is continuously differentiable on $[0,T]$, thus $u(t,x)=\tilde{u}(t,x) + a(t) \in C^{1,2}(\cC)$.}
\end{proof}

{ 
\begin{lemma}\label{lemma:jumps_in_b}
Let $\alpha\in(0,2)$ and $\theta\in\mathcal E$. If, for some $t_0\in(0,T]$, $b_{t_0}<b_{t_0-}$, then
\begin{equation}\label{eq:jumps_in_b}
    b_{t_0-}=\alpha(2-\alpha)\theta_{t_0-}+\rho K.
\end{equation}
\end{lemma}
\begin{proof}
Let $t_0\in(0,T]$ be such that $b_{t_0}<b_{t_0-}$.
Choose $t_0' < t_0$ such that $b_{t_0'} \geq b_{t_0-}$, and $x_1$ and $x_2$ such that $b_{t_0} < x_1 < x_2 < b_{t_0 -} $ . Then, for any $(t, x) \in \mathcal{R} \coloneqq (t_0', t_0) \times (x_1, x_2)$, we have $x < b_t$, and $(t, x) \in \cC$.
Take any $\psi \geq 0 $ in $C^{\infty}_c((x_1, x_2))$.
By $(i)$ of Lemma \ref{freeb:lemma}, the map $t\mapsto v(t,x)$ is non-decreasing for every $x \in \R$. Since $\psi\geq0$, the map $t\mapsto\int_{x_1}^{x_2}v(t,y)\psi(y) dy$ is non-decreasing as well, and, by $(v)$ of Lemma \ref{freeb:lemma}, it is locally absolutely continuous on $(t_0',t_0)$.
By relying on \eqref{eq:PDE}, integrating over $(x_1, x_2)$, we have
\begin{equation} \label{eqn:freeb:prop_proof_eqn1}
\begin{aligned}
    0 & \leq \frac{d}{dt} \int_{x_1}^{x_2}v(t,y)\psi(y) dy = \int_{x_1}^{x_2} \partial_t v(t, y) \psi(y) dy \\
    &= - \int_{x_1}^{x_2} \Big(  y - \alpha (2 - \alpha) \theta_t - \rho v(t, y) \Big)\psi(y) dy  - \frac{1}{2} \sigma^2 \int_{x_1}^{x_2} \partial_{xx} v(t, y) \psi(y) dy \\
    &= \int_{x_1}^{x_2} \Big( \alpha (2 - \alpha) \theta_t + \rho v(t, y) - y \Big)\psi(y) dy  - \frac{1}{2} \sigma^2 \int_{x_1}^{x_2} v(t, y) \psi''(y) dy 
\end{aligned}
\end{equation}
for a.e. $t \in (t_0', t_0)$, where we integrated by parts twice the term on the right-hand side.
Choose a sequence $(t_n)_{n\geq 1} \subset (t_0',t_0)$ such that $t_n\uparrow t_0$ and \eqref{eqn:freeb:prop_proof_eqn1} holds for every $n$. Since $v$ is continuous and bounded on the compact set $\overline{\cR}$, $\theta$ is càdlàg and bounded, and $v(t_0,y)=K$ for every $y\in[x_1,x_2]$, we can let $n\to\infty$ in \eqref{eqn:freeb:prop_proof_eqn1} and obtain
\begin{equation}\label{eqn:freeb:prop_proof_eqn1_bis}
    0\leq\int_{x_1}^{x_2}\left(\alpha(2-\alpha)\theta_{t_0-}+\rho K-y\right)\psi(y)\,dy-\frac{1}{2}\sigma^2K \int_{x_1}^{x_2}\psi''(y)\,dy.
\end{equation}
Since $\psi\in C_c^\infty((x_1,x_2))$, the second integral in \eqref{eqn:freeb:prop_proof_eqn1_bis} vanishes.
Moreover, since inequality \eqref{eqn:freeb:prop_proof_eqn1_bis} holds for every nonnegative $\psi\in C_c^\infty((x_1,x_2))$, it follows that
\[
0 \leq \alpha(2-\alpha)\theta_{t_0-}+\rho K-y ,\qquad y\in(x_1,x_2).
\]
Since $x_2 < b_{t_0-}$ is arbitrary, we deduce $b_{t_0-}\leq\alpha(2-\alpha)\theta_{t_0-}+\rho K$.
On the other hand, $(i)$ of Lemma \ref{optstop:lemma:properties_b} yields $b_t\geq\alpha(2-\alpha)\theta_t+\rho K$, and thus, taking the limit as $t\uparrow t_0$ gives $b_{t_0-}\geq\alpha(2-\alpha)\theta_{t_0-}+\rho K$.
Combining the last two inequalities, we obtain \eqref{eq:jumps_in_b}.
\end{proof}
}

{  
\begin{corollary} \label{freeb:cor:b_continuous}
Let $\alpha \in (0,2)$ and suppose that $\theta \in \cE$ is continuous. Then, the free boundary function $b$ defined in \eqref{optstop:eqn:bounday_b} is continuous on $[0, T)$.
Moreover, $b_{T-}=\alpha(2-\alpha)\theta_T+\rho K$.
\end{corollary}
\begin{proof}
By $(ii)$ of Lemma \ref{freeb:lemma}, the free boundary $b$ is non-increasing and right-continuous on $[0,T)$. Assume by contradiction that $b$ has a discontinuity at some $t_0\in(0,T)$, so that $b_{t_0}<b_{t_0-}$. By Lemma \ref{lemma:jumps_in_b}, $b_{t_0-}=\alpha(2-\alpha)\theta_{t_0-}+\rho K$.
Since $\theta$ is continuous, $\theta_{t_0-}=\theta_{t_0}$, while $(i)$ of Lemma \ref{optstop:lemma:properties_b} yields $b_{t_0}\geq\alpha(2-\alpha)\theta_{t_0}+\rho K=b_{t_0-}$, contradicting $b_{t_0}<b_{t_0-}$. Hence, $b$ has no discontinuities on $(0,T)$ and, together with its right-continuity at $0$, this proves that $b$ is continuous on $[0,T)$.
Finally, as for $b_{T-}$, notice that, since $b_T=-\infty<b_{T-}$ and $\theta$ is continuous, taking $t_0=T$ in \eqref{eq:jumps_in_b} yields the claim.
\end{proof}
}

{  
\begin{theorem}\label{freeb:thm:theta_continuous}
Let $\alpha\in(0,2)$. The unique fixed point $\theta^*\in\mathcal E$ of the map $\Psi:\mathcal E\to\mathcal E$ is continuous on $[0,T]$.
\end{theorem}

\begin{proof}
By Theorem \ref{application:theo:fixed-point}, the map $\Psi:\mathcal E\to\mathcal E$ admits a fixed point, which is unique by Corollary \ref{app:cor:uniqueness}. Let $\theta^* \in \cE$ denote such a fixed point. We prove that $\theta^*$ is continuous.

Since $\theta^*\in\cE$, the map $t\mapsto\theta_t^*$ is càdlàg and non-increasing. 
Notice that $\theta^*$ is continuous at time $T$, since $\xi^*(\theta^*)$ belongs to $\cB_T$ and therefore does not have a jump at terminal time.
Therefore, it is enough to show that it cannot have a downward jump at any $t_0\in(0,T)$. Arguing by contradiction, assume that, for some $t_0\in(0,T)$, $\theta^*_{t_0-}>\theta^*_{t_0}$.
Since $\theta^*$ is a fixed point of $\Psi$, we have $\theta_t^*=x-\E[\xi_t^*(\theta^*)]$, $t\in[0,T]$.
By observing that $\theta^*_{t_0-}=x-\E[\xi_{t_0-}^*(\theta^*)]$, we have
\begin{equation}\label{freeb:thm:theta_continuous:eqn1}
    0 < \theta^*_{t_0-}-\theta^*_{t_0}=\E[\xi_{t_0}^*(\theta^*)-\xi_{t_0-}^*(\theta^*)].
\end{equation}
Recall the definition of $\xi^*(\theta^*)$ in \eqref{eq:optimal_control}.
We now claim that
\begin{equation}\label{freeb:thm:theta_continuous:eqn3}
    0 < \theta^*_{t_0-}-\theta^*_{t_0} = \E[\xi_{t_0}^*(\theta^*)-\xi_{t_0-}^*(\theta^*)] < b_{t_0-}(\theta^*)-b_{t_0}(\theta^*).
\end{equation}
To this extent, we first notice that 
\begin{equation}\label{freeb:thm:theta_continuous:eqn2}
    0 \leq \xi^*_{t_0}(\theta^*) - \xi^*_{t_0-}(\theta^*) \leq b_{t_0-}(\theta^*)-b_{t_0}(\theta^*)
\end{equation}
$\P$-a.s. Indeed, by continuity of the Brownian paths, it follows that $\xi_{t_0-}^*(\theta^*)\geq(x+\sigma W_{t_0}-b_{t_0-}(\theta^*))^+$.
On the other hand, at time $t_0$ the only additional value entering the running supremum is $(x+\sigma W_{t_0}-b_{t_0}(\theta^*))^+$, and thus
\begin{equation}\label{eq:representation_xi_t0}
    \xi_{t_0}^*(\theta^*)=\max\left\{\xi_{t_0-}^*(\theta^*),(x+\sigma W_{t_0}-b_{t_0}(\theta^*))^+\right\}.
\end{equation}
By exploiting the representation \eqref{eq:representation_xi_t0} and the inequality $(z+c)^+ - z^+ \leq c$ for any $c \geq 0$, we get \eqref{freeb:thm:theta_continuous:eqn2}.
We now show that the second inequality in \eqref{freeb:thm:theta_continuous:eqn2} is strict on a set of positive probability. Consider the event $A:=\{x+\sigma W_{t_0}\leq b_{t_0}(\theta^*)\}$.
Since $t_0>0$, $\sigma>0$, and $b_{t_0}(\theta^*)$ is finite, the random variable $x+\sigma W_{t_0}$ has a non-degenerate Gaussian distribution, and therefore $\P(A)>0$.
On the event $A$, $(x+\sigma W_{t_0}-b_{t_0}(\theta^*))^+=0$.
Hence, from the representation \eqref{eq:representation_xi_t0}, $\xi_{t_0}^*(\theta^*)=\xi_{t_0-}^*(\theta^*)$ on $A$.
Moreover, \eqref{freeb:thm:theta_continuous:eqn1} and \eqref{freeb:thm:theta_continuous:eqn2} imply
\[
0 = \xi_{t_0}^*(\theta^*)-\xi_{t_0-}^*(\theta^*)<b_{t_0-}(\theta^*)-b_{t_0}(\theta^*)\qquad\text{on }A.
\]
Together with the pathwise inequality \eqref{freeb:thm:theta_continuous:eqn2} and the fact that $\P(A) >0$, this gives \eqref{freeb:thm:theta_continuous:eqn3}.

Recall that \eqref{freeb:thm:theta_continuous:eqn3} implies that $b(\theta^*)$ has a jump at time $t_0$.
By $(i)$ of Lemma \ref{optstop:lemma:properties_b} and by Lemma \ref{lemma:jumps_in_b}, we have $b_{t_0-}(\theta^*)-b_{t_0}(\theta^*) \leq \alpha(2-\alpha)(\theta^*_{t_0-} - \theta^*_{t_0})$.
Plugging this last inequality in \eqref{freeb:thm:theta_continuous:eqn3}, we have
\begin{equation}\label{freeb:thm:theta_continuous:eqn4}
    0 < \theta^*_{t_0-}-\theta^*_{t_0} < \alpha(2-\alpha) (\theta^*_{t_0-}-\theta^*_{t_0}).
\end{equation}
Since $\alpha(2-\alpha) \in (0,1]$, \eqref{freeb:thm:theta_continuous:eqn4} makes a contradiction.
\end{proof}
}

\begin{theorem} \label{freeb:Theorem:volterra}
Let $\alpha \in (0,2)$ and { let $\theta \in \cE$ be continuous}. 
{ Define $\overline{b}=(\overline{b}_t)_{t \in [0,T]}$ by setting $\overline{b}_t = b_t$ for $t \in [0,T)$ and $\overline{b}_T = b_{T-} = \alpha(2-\alpha)\theta_T + \rho K$.
}{ Then}
$\overline{b} = (\overline{b}_t)_{t\in [0,T]}$ is a continuous and non-increasing solution to the integral equation
\begin{equation} \label{eq:freeb_volterraeqn}
    \E\left[ \int_0^{T - t} e^{-\rho s} \left( { \overline{b}_t }+ \sigma W_s - \alpha (2 - \alpha) \theta_{s + t}  - \rho K\right) \ind_{\left\{ { \overline{b}_t }+ \sigma W_s  \leq \overline{b}_{s + t}\right\}}ds \right] = 0, \quad { t \in [0, T)}.
\end{equation}
Moreover, ${ \overline{b}=(\overline{b}_t)_{t \in [0,T]}}$ is the unique solution of \eqref{eq:freeb_volterraeqn} in the class of continuous non-increasing functions such that ${  \overline{b}_t} \geq \alpha(2-\alpha)\theta_t + \rho K$ for any $t \in [0,T]$.
\end{theorem}

\begin{proof}
Take $(t, x) \in {  [0, T)} \times \R$. {  Fix $\varepsilon \in (0,T-t)$.} Following a standard localization argument, we define
\begin{equation}
    \tau_n \coloneqq \inf \left\{ s \geq t: \ \int_t^s \left\vert\partial_{x} v\left(u, X^{t, x}_u\right) \right\vert^2 du \geq n \right\},
\end{equation}
where $(X^{t, x}_s)_{s \in [t, T]}$ denotes the solution to the uncontrolled dynamics that starts from $x$ at $t$.
Thanks to $(vi)$, $(iii)$ and $(iv)$ of Lemma \ref{freeb:lemma}, respectively, we have that $\big( \partial_t +\frac{1}{2} \sigma^2 \partial_{xx} - \rho\big)v(t,x)$ is locally bounded, that $x \mapsto v(t, x)$ is concave for any $t \in [0,T]$ and that $t \mapsto \partial_x v(t, b_t \pm)$ is continuous on {  $[0,T-\varepsilon]$. Moreover, by Corollary \ref{freeb:cor:b_continuous}, $\overline{b} = b$ is continuous on $[0,T-\varepsilon]$.}
Thus, we can apply the change-of-variable formula by \cite[Theorem 3.1 and Remark 3.2]{PeskirIto} {  on $[t,(T-\varepsilon)\wedge \tau_n]$} and take the expectation, which yields 
\begin{equation} \label{freeb:proof_theo_eqn2}
    \begin{aligned} 
    e^{-\rho t} v(t, x) &= \E \bigg[ e^{-\rho ((T {  -\varepsilon}) \wedge \tau_n)} v\left((T{  -\varepsilon}) \wedge \tau_n, X^{t, x}_{(T{  -\varepsilon}) \wedge \tau_n} \right) \bigg] \\
    & \quad - \E \left[ \int_t^{(T{  -\varepsilon}) \wedge \tau_n} e^{-\rho s} \left( \partial_t +\frac{1}{2} \sigma^2 \partial_{xx} - \rho\right) v\left(s, X^{t, x}_s \right) \ind_{\left\{ X^{t, x}_s \neq {  \overline{b}_s}\right\}}ds \right] \\
    &= \E \bigg[ e^{-\rho ((T{  -\varepsilon}) \wedge \tau_n)} v\left((T{  -\varepsilon}) \wedge \tau_n, X^{t, x}_{(T{  -\varepsilon}) \wedge \tau_n} \right) \bigg] \\
    &\quad \ + \E \left[\int_t^{(T{  -\varepsilon}) \wedge \tau_n} e^{-\rho s} \left( \left( X^{t, x}_s - \alpha (2 - \alpha) \theta_s \right) \ind_{\left\{ X^{t, x}_s \leq {  \overline{b}_s}\right\}} + \rho K \ind_{\left\{ X^{t, x}_s > {  \overline{b}_s}\right\}} \right) ds\right] \\
    &= \E \bigg[ e^{-\rho ((T{  -\varepsilon}) \wedge \tau_n)} v\left((T{  -\varepsilon}) \wedge \tau_n, X^{t, x}_{(T{  -\varepsilon}) \wedge \tau_n} \right) \bigg] \\
    & \quad + \E \left[ \int_t^{(T{  -\varepsilon}) \wedge \tau_n} e^{-\rho s} \left( \left( X^{t, x}_s - \alpha (2 - \alpha) \theta_s - \rho K\right) \ind_{\left\{ X^{t, x}_s \leq {  \overline{b}_s}\right\}} + \rho K \right)ds\right],
\end{aligned}
\end{equation}
where the second equality follows from $(vi)$ of Lemma \ref{freeb:lemma}. { We first let $n \to \infty$ and then $\varepsilon \downarrow 0$.} Since $\tau_n \to \infty$ as $n \to \infty$, { and since, by continuity of $v$ and $X^{t,x}$, $v(T-\varepsilon,X^{t,x}_{T-\varepsilon}) \to v(T,X^{t,x}_T)$ as $\varepsilon \downarrow 0$,} the dominated convergence theorem implies
\begin{multline} \label{freeb:proof_theo_eqn3}
\E \left[ e^{-\rho T } v\left(T, X^{t, x}_{T} \right) + \int_t^{T} e^{-\rho s} \left( \left( X^{t, x}_s - \alpha (2 - \alpha) \theta_s - \rho K\right) \ind_{\left\{ X^{t, x}_s \leq {  \overline{b}_s}\right\}} + \rho K \right) ds \right] \\
= \E \left[ \int_t^{T} e^{-\rho s} \left( X^{t, x}_s - \alpha (2 - \alpha) \theta_s - \rho K\right) \ind_{\left\{ X^{t, x}_s \leq {  \overline{b}_s}\right\}}ds + K e^{-\rho t}\right],
\end{multline}
since $v\big(T, X^{t, x}_{T} \big) = K$. Putting \eqref{freeb:proof_theo_eqn3} back in \eqref{freeb:proof_theo_eqn2} and multiplying both terms for $e^{\rho t}$, we obtain the desired representation for the value function, i.e.
\begin{multline}\label{eq:freeb:representation_V}
v(t, x) = K + \E \left[ \int_t^{T} e^{-\rho (s - t)} \left( X^{t, x}_s - \alpha (2 - \alpha) \theta_s - \rho K\right) \ind_{\left\{ X^{t, x}_s \leq {  \overline{b}_s}\right\}}ds \right] \\
= K + \E \left[ \int_0^{T - t} e^{-\rho s} \left( x + \sigma W_s - \alpha (2 - \alpha) \theta_{s + t} - \rho K\right) \ind_{\left\{ x + \sigma W_s \leq {  \overline{b}_{s + t}} \right\}}ds \right].
\end{multline}
Exploiting the arbitrariness of $(t, x) \in {  [0, T)} \times \R$, {  for any $t\in[0,T)$ we evaluate \eqref{eq:freeb:representation_V} at $x=\overline{b}_t=b_t$. Since $v(t,b_t)=K$, this implies \eqref{eq:freeb_volterraeqn}.}

Finally, the proof of uniqueness follows from the same arguments as those in the proof of \cite[Theorem 25.3]{peskir2006optimal}. {  More precisely, those arguments imply that any continuous non-increasing solution $c$ in the stated class satisfies $c_t=b_t=\overline{b}_t$ for every $t\in[0,T)$. By continuity of $c$ and $\overline{b}$, we then have $c_T=\lim_{t\uparrow T}c_t=b_{T-}=\overline{b}_T$, and therefore $c=\overline{b}$ on $[0,T]$.}
\end{proof}

\subsubsection{An Iterative Scheme for the Equilibrium}

Finally, we briefly discuss a numerical iterative algorithm for approximating the unknown time-dependent {  continuous extension of the free boundary $\overline{b}$}. This algorithm consists of the following steps. 
First, set $\theta^{(0)} \equiv x$, so that $\theta^{(0)} \in \cE$ and it is continuous.
Then, for $n \in \mathbb{N}_0$, the estimate update is obtained iteratively as:
\begin{itemize}[wide]
\item Given $\theta^{(n)}$, numerically solve the following non-linear integral equation exploiting the monotonicity and bound $(i)$ of Lemma \ref{optstop:lemma:properties_b} on the free boundary:
\begin{equation} \label{eq:freeb:num:step} 
    \mathbb{E}\left[\int_0^{T-t} e^{-\rho s} \Big( {  \overline{b}^{(n+1)}_t} + \sigma W_s - \alpha(2-\alpha)\theta^{(n)}_{s+t} - \rho K\Big) \ind_{\big\{ {  \overline{b}^{(n+1)}_t} + \sigma W_s \leq {  \overline{b}^{(n+1)}_{s+t}} \big\} } ds\right] = 0, 
\end{equation}
{  for $t\in[0,T)$.}
\item Given ${  \overline{b}^{(n+1)}}$, approximate 
\begin{equation}\label{eq:freeb:num:step2} 
    \theta^{(n+1)}_t = x - \mathbb{E}\left[\sup_{s\in [0,t]}\Big(x + \sigma W_s - {  \overline{b}^{(n+1)}_s}\Big)^+\right], \quad t \in [0,T], \quad \theta^{(n+1)}_{0^-}=x, 
\end{equation}
using a standard Monte Carlo approximation.
\end{itemize}
Repeat these steps until the difference between $\theta^{(n+1)}$ and $\theta^{(n)}$ is under a certain tolerance.
Figure \ref{fig:numerics} illustrates the equilibrium {  continuous extension of the free boundary $\overline{b}(\theta^*)$}, the equilibrium average $\theta^*$, as well as one realization of the equilibrium optimal control $\xi^*$.
{  
The following proposition establishes the convergence of the algorithm:
\begin{proposition}\label{freeb:prop:iteration_convergence}
Let $\alpha\in(0,2)$ and let $(\theta^{(n)})_{n\geq0}$ be the sequence generated by the iterative scheme above. Then, the sequence converges to the unique fixed point $\theta^*$ of $\Psi$ $dt$-a.e. on $[0,T]$ and strongly in $L^2([0,T])$.
\end{proposition}

\begin{proof}
We first observe that $\theta^{(n+1)} = \Psi(\theta^{(n)})$, with $\Psi$ defined by \eqref{eq:map_psi}, so that the scheme above indeed generates the iterates of the map $\Psi$.
Indeed, by construction, $\theta^{(0)}\equiv x$ belongs to $\cE$ and it is continuous. Thus, by Theorem \ref{freeb:Theorem:volterra}, the continuous extension $\overline{b}^{(1)}=\overline{b}(\theta^{(0)})$ is the unique continuous non-increasing solution of \eqref{eq:freeb:num:step} satisfying the prescribed lower bound.
Since $\overline b^{(1)}$ is the continuous extension of the free boundary associated with $\theta^{(0)}$, the running supremum inside the update in \eqref{eq:freeb:num:step2} gives indeed the optimal control also at terminal time. Hence, $\theta^{(1)}=\Psi(\theta^{(0)})$ on $[0,T]$.
By an induction argument, it follows that $\theta^{(n)}$ is continuous for every $n \geq 1$, $\overline{b}^{(n+1)}$ correctly identifies the continuous extension $\overline{b}(\theta^{(n)})$ of the free boundary and thus $\theta^{(n+1)}$ defined by \eqref{eq:freeb:num:step2} is indeed equal to $\Psi(\theta^{(n)})$.

We now study the convergence of the iterates. Since $\Psi(\cE)\subseteq\cE$ and $\theta^{(0)}\equiv x$, we have $\theta^{(1)}\leq\theta^{(0)}$. 
Next, since $\alpha \in (0,2)$, the map $\theta \mapsto \Psi (\theta)$ is non-decreasing.
This follows from the monotonicity of $\theta \mapsto \xi^*(\theta)$ (cf. Lemma \ref{lemma:monotonicity_b_xi}).
Hence, by induction,
\[
\theta^{(n+1)}=\Psi(\theta^{(n)})\leq\Psi(\theta^{(n-1)})=\theta^{(n)},\qquad n\geq1.
\]
Since every $\theta^{(n)}$ belongs to $\cE$, it is bounded from below by $\min_{t \in [0,T]}\underline{\theta}_t$. Thus, for every $t\in[0,T]$, the sequence $(\theta^{(n)}_t)_{n\geq0}$ admits a pointwise limit $\bar\theta_t\coloneqq\lim_{n\to\infty}\theta^{(n)}_t$.
By dominated convergence, $\theta^{(n)} \to \bar\theta$ strongly in $L^2([0,T])$.
Since $\theta^{(n+1)}=\Psi(\theta^{(n)})$, we also have $\Psi(\theta^{(n)})\to\bar\theta$ strongly in $L^2([0,T])$. The closed-graph property of $\Psi$, established in the proof of Theorem \ref{application:theo:fixed-point}, therefore yields $\Psi(\bar\theta)=\bar\theta$ in $L^2([0,T])$. By uniqueness of the fixed point, $\bar\theta=\theta^*$ $dt$-a.e. on $[0,T]$, and hence $\theta^{(n)}\to \theta^*$ strongly in $L^2([0,T])$.
\end{proof}
}

\begin{figure}[htbp]
  \centering
  \includegraphics[width=0.9\textwidth]{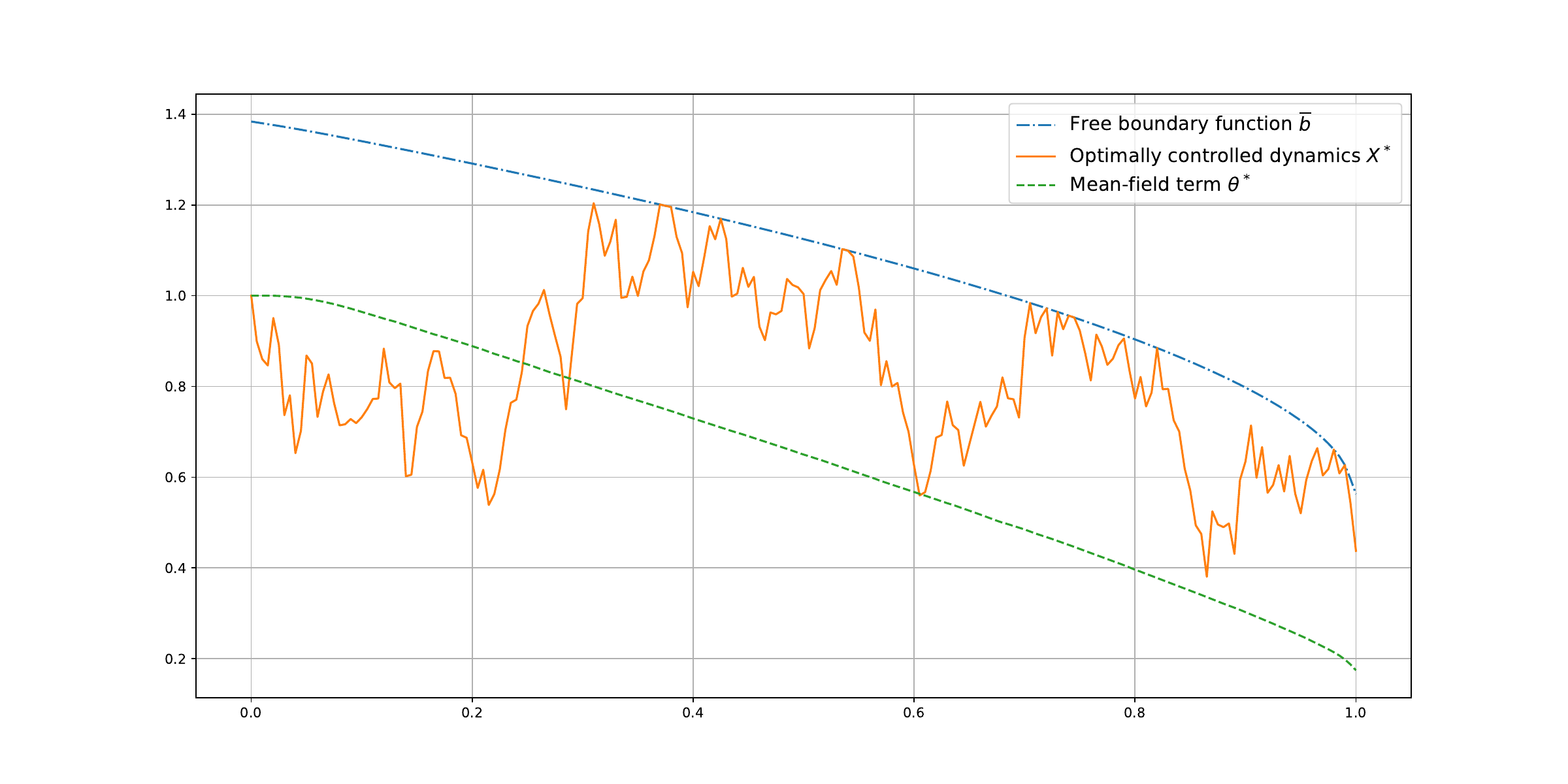}
\end{figure}

\begin{figure}[htbp]
  \centering
  \label{fig:a}
  \includegraphics[width=0.9\textwidth]{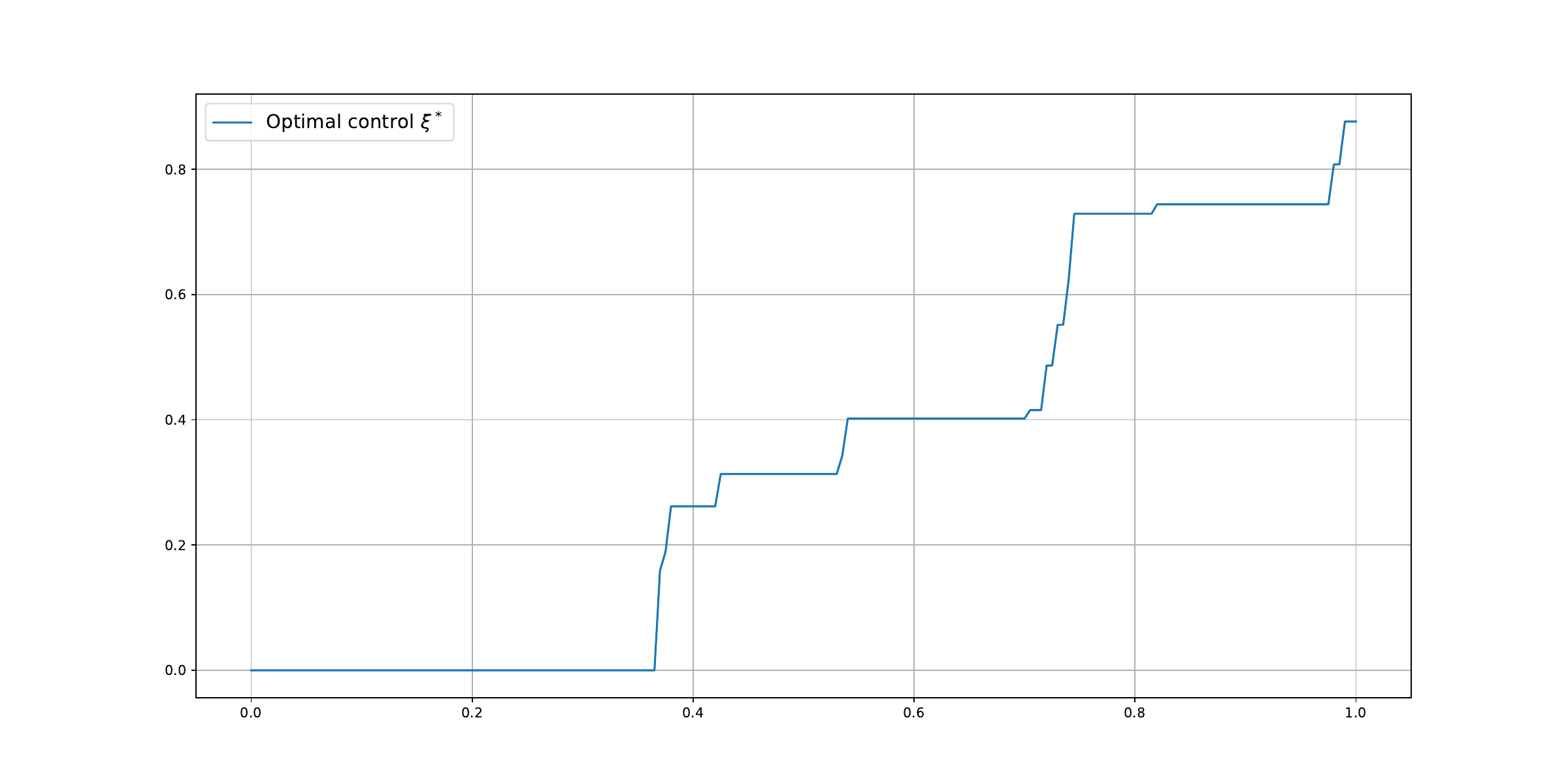}
  \caption{Top: a visual representation of the free boundary function $b$ and the mean-field parameter $\theta^*$, numerically approximated via the iterative algorithm with $T = x = \sigma = K =  1$, $\alpha = 0.2$, and $\rho = 0.5$. A sample of the optimally controlled dynamics $X^*$, obtained with the same hyperparameters, is plotted alongside the two functions. Bottom: a visual representation of the sample path of the optimal control $\xi^*$ associated with the boundary function $b$ and the optimally controlled dynamics $X^*$.} 
  \label{fig:numerics}
\end{figure}

\FloatBarrier
\bibliographystyle{abbrv}
\bibliography{biblio.bib}

@article{hofer2024optimal,
  title={Optimal control and potential games in the mean field},
  author={H{\"o}fer, Felix and Soner, H Mete},
  journal={Stochastic Processes and their Applications},
  pages={104971},
  year={2026},
  publisher={Elsevier}
}

@article{DenkertHorst,
  title={Extended mean-field games with multidimensional singular controls and nonlinear jump impact},
  author={Denkert, Robert and Horst, Ulrich},
  journal={SIAM Journal on Control and Optimization},
  volume={63},
  number={2},
  pages={1374--1406},
  year={2025},
  publisher={SIAM}
}

@article{Fu,
  title={Extended mean field games with singular controls},
  author={Fu, Guanxing},
  journal={SIAM Journal on Control and Optimization},
  volume={61},
  number={1},
  pages={285--314},
  year={2023},
  publisher={SIAM}
}

@article{FuHorst,
  title={Mean field games with singular controls},
  author={Fu, Guanxing and Horst, Ulrich},
  journal={SIAM Journal on Control and Optimization},
  volume={55},
  number={6},
  pages={3833--3868},
  year={2017},
  publisher={SIAM}
}

@article{dianetti2023unifying,
  title={A unifying framework for submodular mean field games},
  author={Dianetti, Jodi and Ferrari, Giorgio and Fischer, Markus and Nendel, Max},
  journal={Mathematics of Operations Research},
  volume={48},
  number={3},
  pages={1679--1710},
  year={2023},
  publisher={INFORMS}
}

@Book{carmona2018probabilistic_vol1,
  author    = {Carmona, Ren{\'e} and Delarue, Fran{\c{c}}ois},
  title     = {Probabilistic Theory of Mean Field Games with Applications. I},
  series    = {Probability Theory and Stochastic Modelling},
  volume    = {83},
  note      = {Mean field FBSDEs, control, and games},
  publisher = {Springer},
  address   = {Cham},
  year      = {2018},
}

@Book{carmona2018probabilistic_vol2,
  author    = {Carmona, Ren{\'e} and Delarue, Fran{\c{c}}ois},
  title     = {Probabilistic Theory of Mean Field Games with Applications. II},
  series    = {Probability Theory and Stochastic Modelling},
  volume    = {84},
  note      = {Mean field games with common noise and master equations},
  publisher = {Springer},
  address   = {Cham},
  year      = {2018},
}

@Book{karatzas1991brownian,
  author    = {Karatzas, Ioannis and Shreve, Steven E.},
  title     = {Brownian Motion and Stochastic Calculus},
  series    = {Graduate Texts in Mathematics},
  volume    = {113},
  edition   = {2},
  publisher = {Springer},
  address   = {New York},
  year      = {1991},
}

@Book{san_paolo_baldi,
  author    = {Baldi, Paolo},
  title     = {Stochastic Calculus: An Introduction through Theory and Exercises},
  series    = {Universitext},
  publisher = {Springer},
  address   = {Cham},
  year      = {2017},
}

@Article{dianetti2020nonzero,
  author  = {Dianetti, Jodi and Ferrari, Giorgio},
  title   = {Nonzero-sum submodular monotone-follower games: existence and approximation of Nash equilibria},
  journal = {SIAM Journal on Control and Optimization},
  volume  = {58},
  number  = {3},
  year    = {2020},
  pages   = {1257--1288},
  doi     = {10.1137/19M1238782},
}

@Article{karatzas1984monotone,
  author  = {Karatzas, Ioannis and Shreve, Steven E.},
  title   = {Connections between optimal stopping and singular stochastic control. I. Monotone follower problems},
  journal = {SIAM Journal on Control and Optimization},
  volume  = {22},
  number  = {6},
  year    = {1984},
  pages   = {856--877},
  doi     = {10.1137/0322054},
}

@Book{pham2009continuous,
  author    = {Pham, Huy{\^e}n},
  title     = {Continuous-Time Stochastic Control and Optimization with Financial Applications},
  series    = {Stochastic Modelling and Applied Probability},
  volume    = {61},
  publisher = {Springer},
  address   = {Berlin},
  year      = {2009},
}

@Book{revuz2013continuous,
  author    = {Revuz, Daniel and Yor, Marc},
  title     = {Continuous Martingales and Brownian Motion},
  publisher = {Springer},
  year      = {2013},
}

@Book{peskir2006optimal,
  author    = {Peskir, Goran and Shiryaev, Albert},
  title     = {Optimal Stopping and Free-Boundary Problems},
  publisher = {Springer},
  year      = {2006},
}

@Article{ferrarideangleissasa,
  author  = {De Angelis, Tiziano and Federico, Salvatore and Ferrari, Giorgio},
  title   = {Optimal boundary surface for irreversible investment with stochastic costs},
  journal = {Mathematics of Operations Research},
  volume  = {42},
  number  = {4},
  year    = {2017},
  pages   = {1135--1161},
  doi     = {10.1287/moor.2016.0841},
}

@Article{baldursson1996irreversible,
  author  = {Baldursson, Fridrik M. and Karatzas, Ioannis},
  title   = {Irreversible investment and industry equilibrium},
  journal = {Finance and Stochastics},
  volume  = {1},
  number  = {1},
  year    = {1996},
}

@Article{dianettiferrarifischer,
  author  = {Dianetti, Jodi and Ferrari, Giorgio and Fischer, Markus and Nendel, Max},
  title   = {Submodular mean field games: existence and approximation of solutions},
  journal = {Annals of Applied Probability},
  volume  = {31},
  number  = {6},
  year    = {2021},
  pages   = {2538--2566},
  doi     = {10.1214/20-aap1655},
}

@article{LasryLions,
  title={Mean field games},
  author={Lasry, Jean-Michel and Lions, Pierre-Louis},
  journal={Japanese Journal of Mathematics},
  volume={2},
  number={1},
  pages={229--260},
  year={2007},
  publisher={Springer}
}

@article{CannerozziFerrari,
  title={Cooperation, Correlation, and Competition in Ergodic N-Player Games and Mean-Field Games of Singular Controls: A Case Study},
  author={Cannerozzi, Federico and Ferrari, Giorgio},
  journal={To appear on Mathematics of Operations Research},
  year={2026},
  publisher={INFORMS}
}

@article{DianettiFerrariTzouanas,
  title={Ergodic Mean Field Games of Singular Control with Regime-Switching},
  author={Dianetti, Jodi and Ferrari, Giorgio and Tzouanas, Ioannis},
  journal={SIAM Journal on Control and Optimization},
  volume={64},
  number={3},
  pages={1830--1858},
  year={2026},
  publisher={SIAM}
}

@article{DianettiLQ,
  title={Linear-quadratic-singular stochastic differential games and applications},
  author={Dianetti, Jodi},
  journal={Decisions in Economics and Finance},
  volume={48},
  number={1},
  pages={381--413},
  year={2025},
  publisher={Springer}
}

@article{Bo-etal,
  title={Constrained mean-field control with singular control: Existence, stochastic maximum principle and constrained FBSDE},
  author={Bo, Lijun and Wang, Jingfei and Yu, Xiang},
  journal={arXiv preprint arXiv:2501.12731},
  year={2025}
}

@article{ShiWu,
  title={Maximum principle for optimal control problems of extended mean-field forward--backward regime-switching systems with general singular controls},
  author={Shi, Hongyu and Wu, Zhen},
  journal={Systems \& Control Letters},
  volume={204},
  pages={106216},
  year={2025},
  publisher={Elsevier}
}

@article{Hafayed2,
  title={On optimal singular control problem for general McKean-Vlasov differential equations: necessary and sufficient optimality conditions},
  author={Hafayed, Mokhtar and Meherrem, Shahlar and Eren, {\c{S}}aban and Gu{\c{c}}oglu, Deniz Hasan},
  journal={Optimal Control Applications and Methods},
  volume={39},
  number={3},
  pages={1202--1219},
  year={2018},
  publisher={Wiley Online Library}
}

@article{Hafayed1,
  title={A mean-field necessary and sufficient conditions for optimal singular stochastic control},
  author={Hafayed, Mokhtar},
  journal={Communications in Mathematics and Statistics},
  volume={1},
  number={4},
  pages={417--435},
  year={2013},
  publisher={Springer}
}

@article{DenkertHorst-MFC,
  title={Extended mean-field control problems with multi-dimensional singular controls},
  author={Denkert, Robert and Horst, Ulrich},
  journal={To appear on The Annals of Applied Probability},
  year={2026}
}

@article{Christensen-etal,
  title={Two sided ergodic singular control and mean-field game for diffusions},
  author={Christensen, S{\"o}ren and Mordecki, Ernesto and Oli{\'u}, Facundo},
  journal={Decisions in Economics and Finance},
  volume={48},
  number={1},
  pages={241--267},
  year={2025},
  publisher={Springer}
}

@article{GuoXu,
  title={Stochastic games for fuel follower problem: N versus mean field game},
  author={Guo, Xin and Xu, Renyuan},
  journal={SIAM Journal on Control and Optimization},
  volume={57},
  number={1},
  pages={659--692},
  year={2019},
  publisher={SIAM}
}

@article{Ferrari-Tzouanas,
  title={Stationary Mean-Field Games of Singular Control under Knightian Uncertainty},
  author={Ferrari, Giorgio and Tzouanas, Ioannis},
  journal={arXiv preprint arXiv:2505.08317},
  year={2025}
}

@article{Cao-etal,
  title={Stationary discounted and ergodic mean field games with singular controls},
  author={Cao, Haoyang and Dianetti, Jodi and Ferrari, Giorgio},
  journal={Mathematics of Operations Research},
  volume={48},
  number={4},
  pages={1871--1898},
  year={2023},
  publisher={INFORMS}
}

@article{CaoGuo,
  title={MFGs for partially reversible investment},
  author={Cao, Haoyang and Guo, Xin},
  journal={Stochastic Processes and their Applications},
  volume={150},
  pages={995--1014},
  year={2022},
  publisher={Elsevier}
}

@article{Guo-etal,
  title={It{\^o}’s formula for flows of measures on semimartingales},
  author={Guo, Xin and Pham, Huy{\^e}n and Wei, Xiaoli},
  journal={Stochastic Processes and their applications},
  volume={159},
  pages={350--390},
  year={2023},
  publisher={Elsevier}
}

@article{Cardaliaguet1,
  title={Stable solutions in potential mean field game systems},
  author={Briani, Ariela and Cardaliaguet, Pierre},
  journal={Nonlinear Differential Equations and Applications NoDEA},
  volume={25},
  number={1},
  pages={1},
  year={2018},
  publisher={Springer}
}

@article{Aid-etal,
  title={A stationary mean-field equilibrium model of irreversible investment in a two-regime economy},
  author={A{\"\i}d, Ren{\'e} and Basei, Matteo and Ferrari, Giorgio},
  journal={Operations Research},
  volume={73},
  number={5},
  pages={2351--2374},
  year={2025},
  publisher={INFORMS}
}

@article{Campi-etal,
  title={Mean-field games of finite-fuel capacity expansion with singular controls},
  author={Campi, Luciano and De Angelis, Tiziano and Ghio, Maddalena and Livieri, Giulia},
  journal={The Annals of Applied Probability},
  volume={32},
  number={5},
  pages={3674--3717},
  year={2022},
  publisher={Institute of Mathematical Statistics}
}

@article{Cardaliaguet2,
  title={Learning in mean field games: the fictitious play},
  author={Cardaliaguet, Pierre and Hadikhanloo, Saeed},
  journal={ESAIM: Control, Optimisation and Calculus of Variations},
  volume={23},
  number={2},
  pages={569--591},
  year={2017},
  publisher={EDP Sciences}
}

@article{Graber,
  title={Remarks on potential mean field games},
  author={Graber, P Jameson},
  journal={Research in the Mathematical Sciences},
  volume={12},
  number={1},
  pages={13},
  year={2025},
  publisher={Springer}
}

@Article{PeskirIto,
  author  = {Peskir, Goran},
  title   = {A change-of-variable formula with local time on curves},
  journal = {Journal of Theoretical Probability},
  volume  = {18},
  number  = {3},
  year    = {2005},
  pages   = {499--535},
  doi     = {10.1007/s10959-005-3517-6},
}

@Article{bahlali2007maximum,
  author  = {Bahlali, Seid and Djehiche, Boualem and Mezerdi, Brahim},
  title   = {The relaxed stochastic maximum principle in singular optimal control of diffusions},
  journal = {SIAM Journal on Control and Optimization},
  volume  = {46},
  number  = {2},
  year    = {2007},
  pages   = {427--444},
  doi     = {10.1137/050644744},
}

@article{cannerozzi2026stationary,
  title={Stationary Mean-Field singular control of an Ornstein-Uhlenbeck process},
  author={Cannerozzi, Federico},
  journal={arXiv preprint arXiv:2601.23036},
  year={2026}
}

@article{Dumitrescu-etal,
  title={Entropy regularization in mean-field games of optimal stopping},
  author={Dianetti, Jodi and Dumitrescu, Roxana and Ferrari, Giorgio and Xu, Renyuan},
  journal={arXiv preprint arXiv:2509.18821},
  year={2025}
}

@book{rudin1976principles,
    AUTHOR = {Rudin, Walter},
     TITLE = {Principles of mathematical analysis},
    SERIES = {International Series in Pure and Applied Mathematics},
   EDITION = {Third},
 PUBLISHER = {McGraw-Hill Book Co., New York-Auckland-D\"usseldorf},
      YEAR = {1976},
     PAGES = {x+342},
   MRCLASS = {26-02},
  MRNUMBER = {385023},
}

@book{Borodin,
    AUTHOR = {Borodin, Andrei N. and Salminen, Paavo},
     TITLE = {Handbook of {B}rownian motion---facts and formulae},
    SERIES = {Probability and its Applications},
   EDITION = {Second},
 PUBLISHER = {Birkh\"auser Verlag, Basel},
      YEAR = {2002},
     PAGES = {xvi+672},
      ISBN = {3-7643-6705-9},
   MRCLASS = {60-00 (60H05 60J25 60J55 60J60 60J65)},
  MRNUMBER = {1912205},
MRREVIEWER = {S\'andor\ Cs\"org\H o},
       DOI = {10.1007/978-3-0348-8163-0},
}

\end{document}